\documentclass[12pt]{amsart}
\usepackage{geometry}
\usepackage{graphicx}
\usepackage{amssymb}
\usepackage{amsmath, mathrsfs, stmaryrd}
\usepackage{color}
\newtheorem{theorem}{Theorem}
\newtheorem{definition}[theorem]{Definition}
\newtheorem{lemma}[theorem]{Lemma}
\newtheorem{corollary}[theorem]{Corollary}
\newtheorem{proposition}[theorem]{Proposition}
\newtheorem{remark}[theorem]{Remark}

\newcommand{\R}{\mathbb R}
\numberwithin{theorem}{section}
\numberwithin{equation}{section}
\begin{document}

\title[conserved quantities]{Conserved quantities in general relativity: from the quasi-local level to spatial infinity}
\author{Po-Ning Chen, Mu-Tao Wang, and Shing-Tung Yau}
\begin{abstract}
We define quasi-local conserved quantities in general relativity by using the optimal isometric embedding in \cite{Wang-Yau2} to transplant Killing fields in the Minkowski spacetime back to the 2-surface of interest in a physical spacetime. To each optimal isometric embedding, a dual element of the Lie algebra of the Lorentz group is assigned. Quasi-local angular momentum and quasi-local center of mass correspond to pairing this element with rotation Killing fields and boost Killing fields, respectively. They obey classical transformation laws under the action of the Poincar\'e group. We further justify these definitions by considering their limits as the total angular momentum and the total center of mass of an isolated system. These expressions were derived from the Hamilton-Jacobi analysis of gravitation action and thus satisfy conservation laws. As a result, we obtained an  invariant total angular momentum theorem in the Kerr spacetime. For a vacuum asymptotically flat i!
 nitial data set of order $1$, it is shown that the limits are always finite without any extra assumptions. We also study these total conserved quantities on a family of asymptotically flat initial data sets evolving by the vacuum Einstein evolution equation. It is shown that the total angular momentum is conserved under the evolution. 
For the total center of mass, the classical dynamical formula relating the center of mass, energy, and linear momentum is recovered, in the nonlinear context of initial data sets evolving by the vacuum Einstein evolution equation. The definition of quasi-local angular momentum provides an answer to the second problem in classical general relativity on Penrose's list \cite{Penrose}.

\end{abstract}

\thanks{Part of this work was carried out while all three authors were visiting Department of Mathematics of National Taiwan University and Taida Institute for Mathematical Sciences in Taipei, Taiwan. P.-N. Chen is supported by NSF grant DMS-1308164, M.-T. Wang is supported by NSF grant DMS-1105483 and S.-T. Yau is supported by NSF
grant  PHY-0714648.} 

\date{December 3, 2013, updated January 28, 2014}
\maketitle
\section{Introduction}

Despite the great success of positive mass theorem (see \cite{Arnowitt-Deser-Misner, Schoen-Yau1, Schoen-Yau2, Witten} and the references therein) in the theory of general relativity, there remain several outstanding and challenging problems regarding the notion of mass. The fundamental difficulty is that, unlike any other physical theory, there is no mass density for gravitation. The naive formula that mass is the bulk integral of mass density is ultimately false. For this reason, most study is limited to the total mass of an isolated gravitating system where it is measured as a flux 2-integral at asymptotic infinity. 
However, a quasi-local description of mass is extremely useful because most physical models are finitely extended regions.
In 1982, Penrose \cite{Penrose} proposed a list of major unsolved problems in general relativity, and the first was ``find a suitable quasi-local definition of energy-momentum (mass)". This is a notion that is attached to a spacelike 2-surface in spacetime. It is expected that other conserved quantities such as the angular momentum can be described quasi-locally (this is the second problem on Penrose's list).

In \cite{Wang-Yau1, Wang-Yau2}, the second and third named authors discovered a definition of quasilocal mass that satisfies highly desirable properties. In particular, the quasi-local mass is positive when the ambient spacetime satisfies the dominant energy condition, and vanishes for 2-surfaces in the Minkowski spacetime. Suppose $\Sigma$ is a spacelike 2-surface in a general physical spacetime $N$. The definition depends only on the induced metric $\sigma$ on $\Sigma$ and the mean curvature vector field $H$ of $\Sigma$. The mean curvature vector field $H$ is a unique normal vector field along the surface that is derived from the  variation of surface area. We consider isometric embeddings of the surface into $\R^{3,1}$ with the same induced metric $\sigma$. Among these isometric embeddings, the mean curvature vector field $H$ picks up optimal ones, by means of a fourth order elliptic equation, which best match the physical embedding of $\Sigma$ in $N$. Once an optimal observer, which corresponds to a future unit time-like Killing field $T_0$ in the Minkowski spacetime, is chosen, there is also a canonical gauge to identity the normal bundle of $\Sigma$ in $N$ and the normal bundle of the image of the isometric embedding in $\R^{3,1}$. In this article, we assign quasi-local angular momentum and quasi-local center of mass to each pair of optimal isometric embedding and observer. The idea is to restrict a Killing field $K$ to the image of the isometric embedding in $\R^{3,1}$ and then use the canonical gauge to transplant $K$ back to the physical surface $\Sigma$ in $N$.  This is consistent with the Hamilton-Jacobi method of deriving conserved quantities and the optimal isometric embedding corresponds to the ground state of the surface. Such an optimal isometric embedding may not be unique in general. However, it was proved in \cite{Chen-Wang-Yau2} that an optimal isometric embedding is locally unique if the quasi-local mass density (denoted as $\rho$ later, see \eqref{rho}) is pointwise positive. 

In the rest of the article, we justify the definitions of these conserved quantities. They transform equivariantly with respect to the action of the Lorentz group on the image of the isometric embedding in $\R^{3,1}$. They also obey the classical transformation law when the optimal isometric embedding is shifted by a translating Killing field (this is equivalent to a shift of the center of mass). The new qusi-local angular momentum is consistent with the Komar angular momentum for an axially symmetric 2-surface in an axially symmetric spacetime. When the quasi-local mass is zero, all quasi-local conserved quantities vanish. In particular, this holds true for a 2-surface in the Minkowski spacetime. We remark that there have been several prior attempts to define quasi-local angular momentum, most notably Penrose's definition \cite{Penrose2}.  We refer to the review article of Szabados \cite{Szabados} and references therein. 

There are few criteria for a valid definition of conserved quantities at the quasi-local level. In order to further justify their physical meaning, we study the limit of these conserved quantities at spatial infinity of asymptotically flat spacetime and considered them as total angular momentum and total center of mass of the system.

There are several existing definitions of total angular momentum and  total center of mass at spatial infinity, for example the Arnowitt-Deser-Misner (ADM) angular momentum \cite{Arnowitt-Deser-Misner} (see also Ashtekar-Hansen \cite{Ashtekar-Hansen}) and the total center of mass proposed by Huisken-Yau \cite{Huisken-Yau}, Regge-Teitelboim \cite{Regge-Teitelboim} (Beig-\'OMurchadha \cite{Beig-Omurchadha}), Christodoulou \cite{Christodoulou}, and Schoen \cite{Huang}. (See for example \cite{HSW, Huang} for these definitions.)  Because these definitions involve asymptotically flat Killing fields which depend on the asymptotically flat coordinates, there remain several questions concerning the finiteness, well-definedness, and physical validity of these definitions. In particular, there are various conditions proposed by Ashtekar-Hansen \cite{Ashtekar-Hansen} , Regge-Teitelboim \cite{Regge-Teitelboim}, Chrusciel \cite{Chrusciel} etc. to guarantee the finiteness of the definition of the total angular momentum and total center of mass at spatial infinity. The most well-studied one is perhaps the Regge-Teitelboim condition, under which there are density theorems of Corvino-Schoen \cite{Corvino-Schoen}, uniqueness of center of mass theorems \cite{Huang}, and effective constructions of initial data sets \cite{Corvino-Schoen, HSW}.

 The new definition only depends on the geometric data $(g, k)$ and the foliation of surfaces at infinity, and does not depend on the asymptotically flat coordinate system or the existence of asymptotically Killing field on the initial data set. Besides, the definition gives an element in the dual space of the Lie algebra of the Lorentz group. In particular, the same formula works for total angular momentum and total center of mass, and the only difference is to pair this element with either a rotation Killing field or a boost Killing field.

We prove that  the total angular momentum on any spacelike hypersurface of Kerr spacetime is an invariant. Note that the definition of the new total angular momentum relies only on $(g_{ij}, k_{ij})$ and does not assume a priori knowledge of a rotation Killing field on the physical initial data set. All Killing fields considered are from the reference Minkowski spacetime through the optimal isometric embedding. In particular, this shows the total angular momentum integral on any spacelike hypersurface of the Minkowski spacetime is zero regardless of the asymptotic behavior at infinity. 

The new total angular momentum and new total center of mass vanish on a spacelike hypersurface of the Minkowski spacetime because the quasi-local mass density for a 2-surface in Minkowski spacetime is always zero. It is not clear if this physical validity condition still holds with other definitions. 
For example, there exists asymptotically flat spacelike hypersurface of the Minkowski spacetime with non-zero ADM total angular  momentum \cite{Chen-Huang-Wang-Yau}. We also proved that the new definitions of total angular momentum and total center of mass are always finite on an asymptotically flat initial data set of order $1$. 

 \vskip 10pt
\noindent {\bf Theorem A} (Theorem \ref{finite})
{\it Suppose $(M, g, k)$ is an asymptotically flat initial data set of order $1$ (see Definition \ref{order_one}) satisfying the vacuum constraint equation. The total center of mass $C^i$ and total angular momentum $J_i$ are always finite.}
\vskip 10pt

 More importantly, we study how theses quantities change along the Einstein evolution equation. For a moving particle, the dynamical formula  in classical mechanics $p=m\dot{r}$ holds where $p$ is the linear momentum, $m$ is the mass, and $r$ is the center of mass. 
The Einstein equation $R_{\mu\nu}-\frac{1}{2} R g_{\mu\nu}=8\pi T_{\mu\nu}$ can be formulated as an initial value problem on a initial data set 
$(M, g, k)$ that satisfies the constraint equation. The solution of the problem is a family $(M, g(t), k(t))$ that satisfies the Einstein evolution equation, a second order hyperbolic system. Therefore, a conceivable  criterion for the validity of a given definition
of total center of mass is to check the validity of this formula under the Einstein evolution equation. 
In the case of strongly asymptotically flat initial data on which the total linear momentum vanishes $p(t)=0$,  Christodoulou \cite{Christodoulou} gave a definition of total center of mass and proved that it is conserved under the Einstein evolution equation.  We are able to prove this dynamical formula for the newly defined total center of mass allowing the more general condition $p(t)\not=0$:

\vskip 10pt
\noindent {\bf Theorem B} (Theorem \ref{thm_variation_center})
{\it Suppose $(M, g, k)$ is an asymptotically flat initial data set of order $1$ (see Definition \ref{order_one}) satisfying the vacuum constraint equation. Let $(M, g(t), k(t) )$ be the solution to the initial value problem $g(0)=g$ and $k(0)=k$ for the vacuum Einstein equation with lapse function $N=1+O(r^{-1})$ and shift vector $\gamma=\frac{\gamma^{(-1)}}{r} +O(r^{-2})$.
The total center of mass $C^i(t)$ and total angular momentum $J_i(t)$ of $(M, g(t), k(t))$ satisfy
\[\begin{split}
\partial_t C^i (t)= &  \frac{p^i}{e},\\
\partial_t J_{i} (t) = &  0
\end{split} \]
for $i=1, 2, 3$ where $(e,p^i)$ is the ADM energy momentum of $(M, g(0), k(0))$.}
\vskip 10pt

In fact, the dynamical formula for the newly defined total center of mass and angular momentum holds under a weaker asymptotically flat condition, assuming the 
total center of mass and angular momentum is finite on the hypersurface $t=0$.

\noindent {\bf Theorem C} (Theorem \ref{cor_variation_center})
{\it Suppose $(M, g, k)$ is a vacuum initial data set satisfying 
\begin{equation}
\begin{split}
g=&\delta+O(r^{-1})\\
k=&O(r^{-2})
\end{split}
\end{equation} such that the total center of mass $C^i$ and total angular momentum $J_i$ of $(M, g, k)$ are both finite.
 Let $(M, g(t), k(t) )$ be the solution to the initial value problem $g(0)=g$ and $k(0)=k$ for the vacuum Einstein equation with lapse function $N=1+O(r^{-1})$ and shift vector $\gamma=\gamma^{(-1)} r^{-1}+O(r^{-2})$.
The total center of mass $C^i(t)$ and total angular momentum $J_i(t)$ of $(M, g(t), k(t))$ satisfy
\[
\begin{split}
\partial_t C^i (t)= &  \frac{p^i}{e},\\
\partial_t J_{i} (t) = &  0
\end{split} 
\]
for $i=1, 2, 3$ where $(e,p^i)$ is the ADM energy momentum of $(M, g, k)$.}

There are also several existing definitions of total center of mass in general relativity such as \cite{Regge-Teitelboim, Beig-Omurchadha, Huisken-Yau}. To the best of our knowledge, this is the first time when such a dynamical formula can be established in the context of the Einstein equation. 

A typical example corresponds to a family of boosted slices moving by the vacuum Einstein evolution equation in the Schwarzschild spacetime.
Let $(y^0, y^1, y^2, y^3)$ be the standard isotropic coordinates of Schwarzchild's solution in which the spacetime metric is of the form:
\begin{equation}\label{isotropic_Sch} G_{\alpha\beta}d y^\alpha dy^\beta=-\frac{1}{F^2(\rho)}(dy^0)^2+\frac{1}{G^2(\rho)} \sum_{i=1}^3 (dy^i)^2\end{equation} with
\[F^2(\rho)=\frac{(1+\frac{M}{2\rho})^2}{(1-\frac{M}{2\rho})^2},\,\,\, G^2(\rho)=\frac{1}{(1+\frac{M}{2\rho})^4}\] and $\rho^2=\sum_{i=1}^3 (y^i)^2$.

 Given constants $\gamma>0$ and $\beta$ with $\gamma=1/\sqrt{1-\beta^2}$. Consider a change of coordinate systems from $(y^0, y^1, y^2, y^3)$ to $(t, x^1, x^2, x^3)$ with $y^0=\gamma t+\beta\gamma x^3$, $y^1=x^1$, $y^2=x^2$, and $y^3=\gamma x^3+\beta\gamma t$. The spacetime metric in the new coordinate system $(t, x^1, x^2, x^3)$ is 
\[\begin{split}&(-\frac{1}{F^2(\rho)}+\frac{1}{G^2(\rho)}\beta^2)\gamma^2 dt^2+2\beta\gamma^2(-\frac{1}{F^2(\rho)}+\frac{1}{G^2(\rho)})dt dx^3\\
&+\frac{1}{G^2(\rho)}\left((dx^1)^2+(dx^2)^2\right)+(-\frac{1}{F^2(\rho)} \beta^2+\frac{1}{G^2(\rho)}) \gamma^2 (dx^3)^2\end{split}\] with
$\rho^2=(x^1)^2+(x^2)^2+(\gamma x^3+\beta\gamma t)^2$.
 
The lapse function $ (-\frac{1}{F^2(\rho)}+\frac{1}{G^2(\rho)}\beta^2)\gamma^2$ and shift (co)-vector $\beta\gamma^2(-\frac{1}{F^2(\rho)}+\frac{1}{G^2(\rho)})dx^3$ satisfy the decay condition in the Theorem. 
 It was computed in \cite{Wang-Yau3} that the limit of the quasilocal energy-momentum is the 4-vector $M(\gamma, 0, 0, -\beta\gamma)$. The center of mass is thus $(C^1(t), C^2(t), C^3(t))$ with $\partial_t C^1(t)=\partial_t C^2(t)=0$ and $\partial_t C^3(t)=-\beta$. The initial values of $(C^1(0), C^2(0), C^3(0))$ do not matter as they are coordinate dependent. \footnote{After this paper was completed and posted on the arXiv on December 3, 2013, we received a paper by Nerz \cite{Nerz} in which the change of the Huisken-Yau definition of center of mass was studied. }

\section{Definition of quasi-local conserved quantities}

In this section, we describe in details the quasi-local conserved quantities we proposed. 
Let $\Sigma$ be a closed embedded spacelike 2-surface in a spacetime $N$. We assume the mean curvature $H$ of $\Sigma$ is spacelike. The data used in the definition of an optimal isometric embedding is the triple $(\sigma,|H|,\alpha_H)$ where $\sigma$ is the induced metric, $|H|$ is the norm of the mean curvature vector, and $\alpha_H$ is the connection one-form of the normal bundle with respect to the mean curvature vector
\[ \alpha_H(\cdot )=\langle \nabla^N_{(\cdot)}   \frac{J}{|H|}, \frac{H}{|H|}   \rangle  \]
where $J$ is the reflection of $H$ through the incoming  light cone in the normal bundle.
Given an embedding $X:\Sigma\rightarrow \R^{3,1}$ and a future timelike unit Killing field $T_0$ in $\R^{3,1}$, we consider the projected embedding $\widehat{X}$ into the orthogonal complement of $T_0$, and denote the induced metric, the second fundamental form, and the mean curvature of the image surface $\widehat{\Sigma}$  by $\hat{\sigma}_{ab}$, $\hat{h}_{ab}$, and $\widehat{H}$, respectively. 

The quasi-local energy with respect to $(X, T_0)$ is (see \S 6.2 of \cite{Wang-Yau2})
\[\begin{split}&E(\Sigma, X, T_0)=\frac{1}{8\pi}\int_{\widehat{\Sigma}} \widehat{H} d{\widehat{\Sigma}}-\frac{1}{8\pi}\int_\Sigma \left[\sqrt{1+|\nabla\tau|^2}\cosh\theta|{H}|-\nabla\tau\cdot \nabla \theta -\alpha_H ( \nabla \tau) \right]d\Sigma,\end{split}\] where $\nabla$ and $\Delta$ are the gradient and Laplace operator of $\sigma$ respectively, $\tau=-\langle X, T_0\rangle$ is considered as a function on the 2-surface,
$|\nabla \tau|^2=\sigma^{ab}\nabla_a \tau\nabla_b\tau$, $\Delta \tau=\nabla^a\nabla_a \tau$, and \begin{equation}\label{theta}\theta=\sinh^{-1}(\frac{-\Delta\tau}{|H|\sqrt{1+|\nabla\tau|^2}}).\end{equation}

In relativity, the energy of a moving particle depends on the observer, and the rest mass is the minimal energy seen among all observers.
In order to define quasi-local mass, we minimize quasi-local energy $E(\Sigma, X, T_0)$ among all admissible pairs $(X, T_0)$ which are considered to be quasi-local observers. A critical point corresponds to an optimal isometric embedding which is defined as follows:

\begin{definition} (see \cite{Wang-Yau2})

Given a closed embedded spacelike 2-surface $\Sigma$ in $N$ with $(\sigma,|H|,\alpha_H)$, an optimal isometric embedding is an embedding $X:\Sigma\rightarrow \R^{3,1}$ such that the induced metric of the image surface in $\R^{3,1}$ is $\sigma$, and there exists a $T_0$ such that
$\tau=-\langle X, T_0\rangle$ satisfies

\[   -(\widehat{H}\hat{\sigma}^{ab} -\hat{\sigma}^{ac} \hat{\sigma}^{bd} \hat{h}_{cd})\frac{\nabla_b\nabla_a \tau}{\sqrt{1+|\nabla\tau|^2}}+ div_\sigma (\frac{\nabla\tau}{\sqrt{1+|\nabla\tau|^2}} \cosh\theta|{H}|-\nabla\theta-\alpha_{H})=0
\] where $\nabla_a\nabla_b$ is the Hessian operator with respect to $\sigma$ and $\theta$ is given by \eqref{theta}.
\end{definition}

Denoting the norm of the mean curvature vector and the connection one-form in mean curvature gauge of the image surface of $X$ in $\R^{3,1}$ by $|H_0|$ and $\alpha_{H_0}$, respectively, we have the following equations relating the geometry of the image of the isometric embedding $X$ and 
the image surface $\widehat{\Sigma}$ of $\widehat{X}$ : 
\[\sqrt{1+|\nabla\tau|^2}\widehat{H} =\sqrt{1+|\nabla\tau|^2}\cosh\theta_0|{H_0}|-\nabla\tau\cdot \nabla \theta_0 -\alpha_{H_0} ( \nabla \tau) \]
and
\[   -(\widehat{H}\hat{\sigma}^{ab} -\hat{\sigma}^{ac} \hat{\sigma}^{bd} \hat{h}_{cd})\frac{\nabla_b\nabla_a \tau}{\sqrt{1+|\nabla\tau|^2}}+ div_\sigma (\frac{\nabla\tau}{\sqrt{1+|\nabla\tau|^2}} \cosh\theta_0|{H_0}|-\nabla\theta_0-\alpha_{H_0})=0.
\]

We can then substitute these relations into the expression for $E(\Sigma, X, T_0)$ and the optimal isometric embedding equation and rewrite them in
term of  \begin{equation} \label{rho} \begin{split}\rho &= \frac{\sqrt{|H_0|^2 +\frac{(\Delta \tau)^2}{1+ |\nabla \tau|^2}} - \sqrt{|H|^2 +\frac{(\Delta \tau)^2}{1+ |\nabla \tau|^2}} }{ \sqrt{1+ |\nabla \tau|^2}}. \end{split}\end{equation}

The quasi-local energy in terms of $\rho$ is
\begin{equation}\label{qle} E(\Sigma, X, T_0)=\frac{1}{8\pi}\int_\Sigma \left[\rho(1+|\nabla\tau|^2)+\Delta \tau \sinh^{-1}(\frac{\rho \Delta\tau}{|H_0| |H|})-\alpha_{H_0}(\nabla \tau)+\alpha_H(\nabla \tau)\right] d\Sigma \end{equation}
and the optimal isometric embedding equation in terms of $\rho$ is 
\begin{equation} \label{optimal3}
div_\sigma\left(\rho \nabla \tau - \nabla [ \sinh^{-1} (\frac{\rho \Delta \tau }{|H_0||H|})] - \alpha_{H_0} + \alpha_{H}\right)=0.
\end{equation}
We note that $\rho=-f$ in equation (4.5) of \cite{Chen-Wang-Yau1}.

We check that for a spacelike 2-surface $\Sigma$ in $\R^{3,1}$ with the induced metric and mean curvature vector in $\R^{3,1}$, the embedding
is an optimal isometric embedding with respect to the data. 
Such an optimal isometric embedding may not be unique in general. However, there are several important cases \cite{mt, Chen-Wang-Yau2} in which a solution of the equation is locally or globally energy minimizing and thus locally unique. In particular, it was proved in \cite{Chen-Wang-Yau2} that if $\rho$ is pointwise positive, the corresponding optimal isometric embedding is locally energy-minimizing and thus locally unique. 

$T_0$ is considered the direction of the quasi-local energy-momentum 4-vector. However, note that being an optimal isometric embedding is invariant under the $SO(3,1)$ action, i.e. if $(X, T_0)$ is an optimal isometric embedding, so is $(AX, AT_0)$ for any $A$ in $SO(3,1)$.

We shall define quasi-local conserved quantities with respect to an optimal isometric embedding. Let $(X^0, X^1, X^2, X^3)$ be the standard coordinate system of $\R^{3,1}$. We recall that $K$ is a rotation Killing field if $K$ is the image of $X^i\frac{\partial}{\partial X^j}-X^j\frac{\partial}{\partial X^i}, i<j$ 
under a Lorentz transformation. $K$ is a boost Killing field if $K$ is the image of $X^i\frac{\partial}{\partial X^0}+X^0 \frac{\partial}{\partial X^i}, i=1, 2, 3$ under a Lorentz transformation. 

\begin{definition} \label{ql_conserved}The quasi-local conserved quantity of $\Sigma$ with respect to an optimal isometric embedding $(X, T_0)$ and a Killing field $K$ is 
\begin{equation}\label{qlcq2}\begin{split}&E(\Sigma, X, T_0, K)\\&=\frac{(-1)}{8\pi} \int_\Sigma
\left[ \langle K, T_0\rangle \rho+K^\top\cdot \left(  \rho {\nabla \tau }- \nabla[ \sinh^{(-1)} (\frac{\rho\Delta \tau }{|H_0||H|})]-\alpha_{H_0}  + \alpha_{H} \right)\right]d\Sigma.\end{split}\end{equation}  Suppose $T_0=A(\frac{\partial}{\partial X^0})$ for a Lorentz transformation $A$, then the quasi-local conserved quantities corresponding to $A(X^i\frac{\partial}{\partial X^j}-X^j\frac{\partial}{\partial X^i}), i<j$ are called the quasi-local angular momentum integral with respect
to $T_0$ and the quasi-local conserved quantities corresponding to $A(X^i\frac{\partial}{\partial X^0}+X^0 \frac{\partial}{\partial X^i}), i=1, 2, 3$ are called the quasi-local center of mass integral with respect to $T_0$. 
\end{definition}
In particular, when $K=T_0$, formula  \eqref{qle} for $E(\Sigma, X, T_0)$  is recovered. 
It is straightforward to see that the quasi-local energy of an optimal isometric embedding $(X, T_0)$ is 
\begin{equation}\label{qlm}\frac{1}{8\pi} \int_\Sigma \rho.\end{equation}

\begin{definition}
Let $(X, T_0)$ be an optimal isometric embedding. Then the function $\rho$ defined in \eqref{rho} is called the quasi-local mass density of $(X, T_0)$ and the integral \eqref{qlm} is the quasi-local mass  of $(X, T_0)$. 
\end{definition}
This is analogous to the rest mass in special relativity which is obtained by minimizing energy seen by all observers.

\section{Definition of total conserved quantities at spatial infinity}

Consider an initial data set  $( M, g, k)$ where $M$ is a 3-manifold, $g$ is an Riemannian metric on $M$ and $k$ is a symmetric two-tensor on $M$. 

\begin{definition}\label{order_one}
 $(M, g, k)$ is \textit{asymptotically flat of order one} if there is a compact subset $C$ of $M$ such that 
$ M \backslash C $ is diffeomorphic to $\R^3 \backslash B$. The diffeomorphism induces a ``Cartesian" coordinate system $\{x^i\}_{i=1, 2, 3}$ on $M\backslash C$ and, in terms of this coordinate system, we have the following decay condition for $g$ and $k$.
\begin{equation} 
\begin{split}
g_{ij} & = \delta_{ij}+ \frac{g_{ij}^{(-1)}}{r}+  \frac{g_{ij}^{(-2)}}{r^2}+ o(r^{-2})\\
k_{ij} & =   \frac{k_{ij}^{(-2)}}{r^2}   +  \frac{k_{ij}^{(-3)}}{r^3} +o(r^{-3})        
\end{split}
\end{equation}
where $r=\sqrt{\sum_{i=1}^3 (x^i)^2}$. 
\end{definition}

There is also a spherical coordinate system $\{ r, \theta=u^1, \phi=u^2\}$ system on $M\backslash C$
defined by \[\begin{cases}x^1&=r\sin\theta\sin\phi\\x^2&=r\sin\theta\cos\phi\\x^3&=r\cos\theta\end{cases}.\]

On each level set of $r$,  $\Sigma_r$, we can use $\{u^a\}_{a=1, 2}$ as coordinate system to express the geometric data we need in order to define quasi-local conserved quantities: 
\[  
\begin{split}
\sigma_{ab} & = r^2 \tilde \sigma_{ab}+ r \sigma_{ab}^{(1)} + \sigma_{ab}^{(0)}+ o(1) \\
|H| & = \frac{2}{r}+ \frac{h^{(-2)}}{r^2}+\frac{h^{(-3)}}{r^3} + o(r^{-3}) \\
\alpha_H & = \frac{\alpha_H^{(-1)} }{r}+ \frac{\alpha_H^{(-2)} }{r^2} + o(r^{-2})
\end{split}
\] 
where $\tilde{\sigma}_{ab}$ is the standard metric on a unit round sphere $S^2$. $\sigma_{ab}^{(1)}$ and $ \sigma_{ab}^{(0)}$ are considered as symmetric 2-tensors on $S^2$, $h^{(-2)}$ and $h^{(-3)}$ are functions on $S^2$ and $\alpha_H^{(-1) }$ and  $\alpha_H^{(-2)}  $ are 1-forms on $S^2$. We shall see that all total conserved quantities on $(M, g, k)$ can be determined by these data on $S^2$.

Using the result of \cite{Chen-Wang-Yau1}, there is a unique family of isometric embedding $X_r$ and observer $T_0(r)$ which minimizes the quasi-local energy locally. 
The embedding $ X({r})= (X^0({r}), X^i({r})) $ has the following expansion
\[  
\begin{split}
X^0({r}) & = (X^0)^{(0)} + \frac{(X^0)^{(-1)}}{r} + o(r^{-1}) \\
X^i ({r})& = r \tilde X^i + (X^i)^{(0)} + \frac{(X^i)^{(-1)}}{r} + o(r^{-1}) \\
T_0 ({r})& = (a^0,a^i) + \frac{T_0^{(-1)}}{r} + o(r^{-1}).
\end{split}
\] 
In particular, the image of the isometric embedding $X_r$ approaches the standard round sphere of radius $r$ in $\R^3$.
We note that the gauge choice of the optimal isometric embedding $(X, T_0)$ is fixed by this choice. 

The dependence goes as follows, $X_i^{(0)}$ depends on $\sigma_{ab}^{(1)}$, $h_0^{(-2)}$ depends on $X_i^{(0)}$, and $\tau^{(0)}$ depends on 
$(a^0, a^i)$, $h^{(-2)}$, $h_0^{(-2)}$, and $\alpha_{H}^{(-1)}$, etc.
As a result, the data on the image of $X_r$ also satisfy
\[  
\begin{split}
|H_0| & = \frac{2}{r}+ \frac{h_0^{(-2)}}{r^2}+\frac{h_0^{(-3)}}{r^3} + o(r^{-3}) \\
\alpha_{H_0} & = \frac{\alpha_{H_0}^{(-1)} }{r}+ \frac{\alpha_{H_0}^{(-2)}}{r^2} + o(r^{-2}).
\end{split}
\] 
We recall from \cite{Wang-Yau3} and \cite{Chen-Wang-Yau1},
\[m=\frac{1}{8 \pi}\int_{S^2} \rho^{(-2)}dS^2=\frac{1}{8 \pi a^0} \int_{S^2}  (h_0^{(-2)}-h^{(-2)}) dS^2\]
is the ADM mass,
\begin{equation}\label{ADM_e}\frac{1}{8 \pi}\int_{S^2} (h_0^{(-2)}-h^{(-2)}) dS^2=ma^0=e\end{equation} 
is the ADM energy and 

\begin{equation}\label{ADM_p}\frac{1}{8 \pi}\int_{S^2} \tilde{X}^i \widetilde{div}(\alpha_H^{(-1)})  dS^2=m a^i=p^i\end{equation} is the  ADM linear momentum.

\begin{definition}
Suppose $ T_0({r})=A(r)(\frac{\partial}{\partial X^0})$ for a family of Lorentz transformation $A(r)$.  Define
\begin{equation}
C^i  =  \frac{1}{m}\lim_{r \to \infty} E(\Sigma_r, X_r, T_0(r), A(r)(X^i\frac{\partial}{\partial X^0}+X^0\frac{\partial}{\partial X^i}))\\
\end{equation} to be the total center of mass and 
\begin{equation}
J_{i} =\lim_{r \to \infty} \epsilon_{ijk}E(\Sigma_r , X_r, T_0(r), A(r)(X^j\frac{\partial}{\partial X^k}-X^k\frac{\partial}{\partial X^j})) 
\end{equation} to be the total angular momentum,
where  $\Sigma_r$ are the coordinate spheres and $(X_r, T_0({r}))$ is the unique family of optimal isometric embeddings of $\Sigma_r$ such that $X_r$ converges to a round sphere of radius $r$ in $\R^3$ when $r\rightarrow \infty$.
\end{definition}


\section{Change of quasi-local center of mass and angular momentum with respect to reference frame.}
In special relativity, one can associate energy-momentum 4-vector and conserved quantities for a particle moving along a geodesics. Let $(X^0, X^1, X^2, X^3)$ denote the standard coordinate system on $\R^{3,1}$ such that the Minkowski metric is of the form $\eta_{\alpha\beta}=-(dX^0)^2+\sum_i (dX^i)^2$. Suppose the particle moves along the geodesics $X(s)$. Denote $\partial_s X(s)$ by $\dot {X}(s)$.

 For a Killing vector field $K$ in $\R^{3,1}$, the conserved quantity associates to the geodesics $X(s)$ and the Killing vector field $K$ is 
\[ \dot {X}^{\alpha} \eta_{\alpha \beta} K^{\beta}.  \]

The energy-momentum 4-vector of the particle is obtained using the Killing vector field $K=\frac{\partial}{\partial X^{\alpha}}$. For such $K$, the associate conserved quantity is $\dot {X}_{\alpha}$. The energy of the particle measured by an observers $T_0$ can be considered as a function $e_X$ on the set of future-directed unit timelike vectors, $\mathbb H^3$ in $\R^{3,1}.$ Moreover, $e_X$ is the restriction of a linear function on $\R^{3,1}$ to $\mathbb H^3$. The  energy-momentum 4-vector $\dot{X}$ is  the dual of this linear function  $e_X$. For a family of  observer $T_0(t)$ in $\mathbb H^3$ with $T_0(0)=\frac{\partial}{\partial X^0}$, 
\begin{equation} \label{newtonian1}
    \partial_t e_X(T_0(t))|_{t=0} = a^i \dot{X}_i \end{equation}
where
\[
 \partial_t T_0|_{t=0}= -a^i  \frac{\partial}{\partial X^i}.
\]
Furthermore, if we compute the conserved quantity for the Killing vector field  $K=K_{\alpha}^{\,\,\,\,\beta} X^\alpha\frac{\partial}{\partial X^\beta}$ where $K_{\alpha \beta}$ is skew-symmetric, then the conserved quantity become 
\[ \dot {X}^{\alpha} g_{\alpha \beta} K^{\beta} = \dot {X}^{\alpha} K_{\alpha\beta} X^\beta. \]

Moreover, if we translate the geodesics by a constant vector $b=(b^{\alpha})$, we obtain a new geodesics $(X')^{\alpha}=X^{\alpha}+ b^{\alpha}$. The difference of the conserved quantities along the two geodesics $X$ and $X'$ is
\begin{equation} \label{newtonian2}
  \dot {X}^{\alpha}K_{\alpha \beta}b^\beta 
\end{equation} 

In this section, we prove equation  \eqref{newtonian1} and \eqref{newtonian2} for quasi-local energy-momentum, center of mass and angular momentum. 

\subsection{Definition of quasi-local conserved quantities with respect to a general isometric embedding}

Quasi-local conserved quantities can be defined for $(X, T_0)$ that is not necessarily optimal and they also transform well under  the Lorentz group action on $\R^{3,1}$. 
Let  $K=K_{\alpha}^{\,\,\,\,\beta}  X^\alpha\frac{\partial}{\partial X^\beta}$ denote a Killing field, the corresponding conserved quantity is 

\[\frac{(-1)}{8\pi} K_{\alpha \gamma} \int_\Sigma X^\alpha \left( \rho T_0^\gamma+(\nabla_a X^\gamma )j^a \right) d\Sigma
\]
where $X^\alpha=X^\alpha(u^a)$ are components of the optimal isometric embedding, $K_{\alpha\gamma}=K_{\alpha}^{\,\,\,\,\beta}  \eta_{\beta\gamma}$, and 
\begin{equation} \label{optimal_embedding_vector}
j^a=\sigma^{ac}\left( \rho {\partial_c \tau }- \partial_c [ \sinh^{-1} (\frac{\rho\Delta \tau }{|H_0||H|})]-(\alpha_{H_0})_c  + (\alpha_{H})_c \right). \end{equation}   Note that $K_{\alpha\gamma}=-K_{\gamma\alpha}$ is skew-symmetric. 

We interpret \[p^\gamma=\frac{1}{8\pi}  \int_\Sigma \left( \rho T_0^\gamma+(\nabla_a  X^\gamma) j^a \right) d\Sigma, \gamma=0, 1, 2, 3\] as the energy-momentum 4-vector and thus  \[e=\frac{1}{8\pi}  \int_\Sigma \left( \rho T_0^0+(\nabla_a X^0) j^a \right) d\Sigma\] is the energy and \[p^i=\frac{1}{8\pi}  \int_\Sigma \left( \rho T_0^i+(\nabla_a X^i) j^a \right) d\Sigma, i=1, 2, 3\] corresponds to the linear momentum.

Therefore, we can write the quasi-local conserved quantity as an element $\Phi^{\alpha\gamma}$ in the dual space of $K_{\alpha\gamma}$:

\[\Phi^{\alpha\gamma}=\frac{-1}{16\pi} \int_\Sigma \left[ (X^\alpha  T_0^\gamma- X^\gamma  T_0^\alpha)\rho +[X^\alpha( \nabla_a X^\gamma)- X^\gamma (\nabla_a X^\alpha)]j^a \right]d\Sigma \] where $\Phi^{\alpha\gamma}=-\Phi^{\gamma\alpha}$.  

\subsection{Change of quasi-local center of mass and angular momentum with respect to reference frame.}
First we compute the change of quasi-local center of mass and angular momentum with respect to reference frame. This is analogous to equation  \eqref{newtonian2}.
\begin{theorem} \label{general1}
Suppose $X$ is an isometric embedding of the surface $\Sigma$ into $\R^{3,1}$. Let $b$ be a vector in $\R^{3,1}$ and consider the isometric embedding $X'=X+b$. Let $p^\gamma$ be the quasi-local energy-momentum 4-vector with respect to the isometric embedding $X$ and  $\Phi^{\alpha\gamma}$ and $\Phi'^{\alpha\gamma}$ be the quasi-local conserved quantities with respect to $X$ and $X'$. Then we have
\[\Phi'^{\alpha\gamma}=\Phi^{\alpha\gamma}-\frac{b^\alpha p^\gamma}{2}+\frac{b^\gamma p^\alpha}{2}.\]
\end{theorem}
\begin{proof}
As $|H_0|$, $\alpha_{H_0}$ are all invariant under the Lorentzian group action on the image of the isometric embedding $X:\Sigma\rightarrow \R^{3,1}$, $\rho$ and $j^a$ are invariantly defined.  From the expression, $\Phi^{\alpha\gamma}$ is equivariant with respect to the $SO(3,1)$ action on $X$ and $T_0$.  Suppose the isometric embedding $X$ is shifted by $(X')^\alpha = X^\alpha+b^\alpha$ for a constant 4-vector $b^\alpha$. The new 
$\Phi'^{\alpha\gamma}$ is given by 
\[
\begin{split}\Phi'^{\alpha\gamma}-\Phi^{\alpha\gamma}
= & \frac{-1}{16\pi} \int_\Sigma \left[ (b^\alpha  T_0^\gamma- b^\gamma  T_0^\alpha)\rho +(b^\alpha \nabla_a X^\gamma- b^\gamma\nabla_aX^\alpha)j^a \right]d\Sigma \\
= &  \frac{1}{2}(-b^\alpha p^\gamma+b^\gamma p^\alpha).
\end{split}
\]
\end{proof}

Next, we relate the quasi-local linear momentum $p^i$ to the variation of quasi-local energy. This is analogous to equation  \eqref{newtonian1}.
\begin{theorem}\label{general2}
Consider the function 
\[ f(s) = E(\Sigma, X,  \sqrt{1+s^2}\frac{\partial}{\partial X^0} + s \frac{\partial}{\partial X^i}, \sqrt{1+s^2}\frac{\partial}{\partial X^0} + s \frac{\partial}{\partial X^i} )  \]
That is, $f(s)$ is the quasi-local energy with respect to the embedding $X$ and the observer $\sqrt{1+s^2}\frac{\partial}{\partial X^0} + s \frac{\partial}{\partial X^i}$. 
Then
\[  f'(0) =  E(\Sigma, X,  \frac{\partial}{\partial X^0},  \frac{\partial}{\partial X^i} )\] 
\end{theorem}
\begin{proof}
Recall that the expression of $E$ is given in equation \eqref{qlcq2}. Hence, we have
\[
E(\Sigma, X,  \frac{\partial}{\partial X^0},  \frac{\partial}{\partial X^i} ) =\frac{1}{8\pi} \int_\Sigma
 (\frac{\partial}{\partial X^i})^\top\cdot \left(  \rho {\nabla \tau }- \nabla[ \sinh^{-1} (\frac{\rho\Delta \tau }{|H_0||H|})]-\alpha_{H_0}  + \alpha_{H} \right ) d \Sigma.
\]
Moreover, we have
\[ 
 (\frac{\partial}{\partial X^i})^\top =\nabla X^i.
 \]
As a result, 
\[
\begin{split}
   E(\Sigma, X,  \frac{\partial}{\partial X^0},  \frac{\partial}{\partial X^i} ) 
=& \frac{1}{8\pi} \int_\Sigma  \nabla_a X^i j^a  d \Sigma\\
= & \frac{1}{8\pi} \int_\Sigma  (-X^i)   \nabla_a j^a d \Sigma.
\end{split}
\]
On the other hand,  $f(s)$ corresponds to the quasi-local energy with respect to the same embedding $X$ with the observer $T(s)= \sqrt{1+s^2}\frac{\partial}{\partial X^0} + s \frac{\partial}{\partial X^i}$. Let $ \tau(s) $ be the time function with respect to embedding $X$ and  the observer $T(s)$.
We have
\[   \tau(s) = \sqrt{1+s^2}\tau(0) - s X^i. \]
In particular, 
\[  \partial_s  \tau(s) |_{s=0} =-X^i.   \]
Moreover, we have
\[
-(\widehat{H}\hat{\sigma}^{ab} -\hat{\sigma}^{ac} \hat{\sigma}^{bd} \hat{h}_{cd})\frac{\nabla_b\nabla_a \tau}{\sqrt{1+|\nabla\tau|^2}}+ div_\sigma (\frac{\nabla\tau}{\sqrt{1+|\nabla\tau|^2}} \cosh\theta|{H}|-\nabla\theta-\alpha_{H}) 
= \nabla_a j^a
 \]
and thus, by the first variation formula of quasi-local energy in \cite{Wang-Yau2}, it follows that 
\[ f'(0)=  \frac{1}{8\pi} \int_\Sigma ( -X^i)    \nabla_a j^a d \Sigma.  \]
\end{proof}
From Theorem \ref{general2}, we have
\begin{corollary}
Given a surface $\Sigma$ in $N$, suppose $(X,T_0)$ is a solution of the optimal embedding equation, then the quasi-local linear momentum $p^i$ vanishes.
\end{corollary}
Combining Theorem \ref{general1} and the above corollary, we have
\begin{corollary}
Given a surface $\Sigma$ in $N$, suppose the pair $(X,T_0)$ is a solution of the optimal embedding equation, then for any constant vector $b$, the angular momentum with respect to the embedding $X$ is the same as that of $X'=X+b$.
\end{corollary}

\section{Quasi-local conservation law}

An important ingredient of our proof of the invariance of total angular momentum is a conservation law of the quasi-local conserved quantity. In view of the definition of $\rho$,
the quasi-local conserved quantity can be written as the difference of a reference term and a physical term. 

\begin{equation}\begin{split}&E(\Sigma, X, T_0, K)\\
=&\frac{(-1)}{8\pi} \int_\Sigma
\{\langle K, T_0\rangle \frac{(\cosh\theta_0 |H_0|-\cosh\theta |H|)}{\sqrt{1+|\nabla\tau|^2}}\\
+&K^\top\cdot [  \frac{\nabla \tau (\cosh\theta_0 |H_0|-\cosh\theta |H|)}{\sqrt{1+ |\nabla \tau|^2}}- \nabla(\theta_0 - \theta) -\alpha_{H_0} + \alpha_{H} ]\}d\Sigma.
\end{split}\end{equation} 

Given $(X, T_0)$, denote the image of $X$ in $\R^{3,1}$ by $\Sigma$. We recall that $\widehat{\Sigma}$ is the 2-surface that is the projection of the image of the isometric embedding $X$ onto the orthogonal complement of $T_0$. Let $\mathcal{C}$ be the timelike cylinder that is generated by $T_0$ along $\widehat{\Sigma}$. Both $\Sigma$ and $\widehat{\Sigma}$ are contained in $\mathcal{C}$. 
The reference term is an integral on the image of the isometric embedding that can be rewritten as (up to a scalar multiple)
\[\begin{split}&\int_\Sigma
\{\langle K, T_0\rangle \frac{(\cosh\theta_0 |H_0|)}{\sqrt{1+|\nabla\tau|^2}}
+K^\top\cdot [  \frac{\nabla \tau (\cosh\theta_0 |H_0|)}{\sqrt{1+ |\nabla \tau|^2}}- \nabla \theta_0  -\alpha_{H_0} ]\}d\Sigma\\
&=\int_\Sigma
[\langle K, \breve{e}_4\rangle (\cosh\theta_0 |H_0|)
-K^\top\cdot \left( \nabla \theta_0 +\alpha_{H_0} \right)]d\Sigma\end{split}\] where $\breve{e}_4= \frac{1}{\sqrt{1+|\nabla\tau|^2}}(T_0+\nabla\tau)$.

If $K$ is tangential to $\mathcal{C}$, this can be further simplified as 
\[\int_\Sigma
\pi(K, \breve{e}_4) d\Sigma \] where $\pi$ is the conjugate momentum of the timelike cylinder $\mathcal{C}$ with respect to $\breve{e}_3$, the unit outward spacelike unit normal of $\mathcal{C}$. Note that $\breve{e}_4$ is the timelike unit normal of $\Sigma$ in $\mathcal{C}$, and $\breve{e}_3$ and $\breve{e}_4$ are orthogonal along $\Sigma$. 

Since $\pi$ satisfies the vacuum constraint equation  and $K$ is a Killing field, $\pi(K, \cdot)$ is a divergence free on $\mathcal{C}$ and thus by applying the divergence theorem to $\Omega$, the portion of $\mathcal{C}$ bounded by $\Sigma$ and $\widehat{\Sigma}$, we obtain
\[\int_\Sigma
\pi(K, N) d\Sigma =\int_{\widehat{\Sigma}} \pi (K, \widehat{N}) d\hat{\Sigma}\]  where $N$ and $\widehat{N}$ are the future timelike unit normal of $\Sigma$ and $\widehat{\Sigma}$, respectively as boundary components of $\Omega$.

As $\widehat{N}=T_0$ on $\widehat{\Sigma}$, this is equal to 
\[\int_{\widehat{\Sigma}} \pi (K, \breve{e}_4) d\hat{\Sigma}=\int_{\widehat{\Sigma}} \hat{k}\langle K, T_0 \rangle d\hat{\Sigma}.\]

We can also apply this to the physical term
\[\int_\Sigma\{
\langle K, T_0\rangle \frac{(\cosh\theta |H|)}{\sqrt{1+|\nabla\tau|^2}}
+K^\top\cdot [ \frac{\nabla \tau (\cosh\theta |H|)}{\sqrt{1+ |\nabla \tau|^2}}- \nabla \theta  -\alpha_{H}]\}d\Sigma\] and relate this to another surface term on a time slice in the physical spacetime. There will be error terms in applying this formula, but they often can be controlled under various conditions. 

We remark that this conservation law was observed by Brown-York \cite{Brown-York1, Brown-York2}. The novelty here is that this law is applied to both the physical and reference Minkowski spacetime. 

\section{Quasi-local conserved quantities in an axially symmetric spacetime}
For a spacetime with symmetry, it is useful to construct conserved quantity using the Killing vector field. For an axially symmetric vacuum spacetime, this procedure gives the well-known Komar angular momentum. In this section, we show that for an axially symmetric vacuum spacetime $N$, our quasi-local angular momentum agrees with the Komar angular momentum. For completeness, we first recall the Komar construction. Suppose $\eta$ is a Killing vector field in $N$. Then
\[  \nabla^N_\alpha  \eta_{\beta} + \nabla^N_\beta  \eta_{\alpha}  = 0, \]
it follows that $\nabla^N_\alpha  \eta_{\beta} $ is a two-form. Let $S_{\alpha\beta}=\star \nabla^N_\alpha  \eta_{\beta} $ be its Hodge dual. In a vacuum spacetime, one can verify that 
$S_{\alpha\beta}$ is closed and co-closed. As a result, the integral 
\[  \frac{-1}{8 \pi} \int_{\Sigma} S_{\alpha \beta} \]
depends only on the homology class of $\Sigma$. 

For an axially symmetric vacuum spacetime, the Komar angular momentum is obtained by applying the above construction to the rotation Killing vector field $\frac{\partial}{\partial \phi}$. In particular, for an axially symmetric surface $\Sigma$  in an axially symmetric spacetime, the Komar angular momentum reduces to 
\[ J= \frac{1}{8 \pi}\int_{\Sigma} \langle \nabla^N_{\frac{\partial}{\partial \phi}} e_3, e_4 \rangle d\Sigma \]
where $\{  e_3,e_4 \}$ is any frame of the normal bundle of $\Sigma$ with $e_3$ is an outward spacelike unit normal and $e_4$ is a future directed timelike unit normal. This expression is independent of the choice of frame because $\frac{\partial}{\partial \phi}$ is Killing. 

\begin{proposition}
Let $\Sigma$ be an axially symmetric surface in an axially symmetric spacetime. Let $\tau$ be an axially symmetric function on $\Sigma$.  Consider the isometric embedding $X$ of $\Sigma$ into $\R^{3,1}$ with time function $\tau=-\langle X, T_0\rangle$. Let $K$ be the rotation Killing field in $\R^{3,1}$ that corresponds to the symmetry on the image of the isometric embedding $X$. Then 
\[  E(\Sigma, X, T_0 , K) = J.\]  
\end{proposition}
\begin{proof} Since the 2-surface $\widehat{\Sigma}$ in the orthogonal complement of $T_0$ is also axially-symmetric, $K$ and $T_0$ are orthogonal. 
Recall that the angular momentum with respect to $K$ is
\[\int_{\Sigma}[K^\top\cdot (  \rho {\nabla \tau }- \nabla[ \sinh^{-1} (\frac{\rho\Delta \tau }{|H_0||H|})]-\alpha_{H_0}  + \alpha_{H})]d\Sigma.\]
 By assumption, the surface and the time function are both axially symmetric, it follows that 
$K$ is tangent to the image of $X$ and its pull back to the surface in the physical spacetime is the restriction of the rotation Killing vector field $\frac{\partial}{\partial \phi}$. 

From the discussion in 
Section 4 and that $K$ is a Killing vector field in $\R^{3,1}$, the reference term is
$\int_{\Sigma} \pi (K, \breve e_4) d \Sigma$. Moreover, 

\[
\int_{\Sigma} \pi (K, \breve e_4) d \Sigma =\int_{\widehat{\Sigma}} \pi (K, T_0 ) d\widehat{\Sigma}.
\]  
However, since $\widehat{\Sigma}$ lies in a totally geodesic hyperplane, we have
\[ \pi (K, T_0)=0  \]
and it follows that the reference term is zero. 

On the other hand, the physical term is simply
\[ -\frac{1}{8 \pi}\int_{\Sigma}\alpha_H(\frac{\partial}{\partial \phi})  d \Sigma,  \]
since the time function $\tau$ and the angle $\theta$ are both independent of $\phi$.
As a result, 
\[ E(\Sigma, X, T_0 , K) =  \frac{1}{8 \pi}\int_{\Sigma}\alpha_H(\frac{\partial}{\partial \phi}) d \Sigma  \]
 This finishes the proof of the lemma since 
\[ J= \frac{1}{8 \pi}\int_{\Sigma} \langle \nabla_{\frac{\partial}{\partial \phi}} e_3, e_4 \rangle  d \Sigma \]
where $\{  e_3,e_4 \}$ is any frame of the normal bundle of $\Sigma$ with $e_3$ is an outward spacelike unit normal and $e_4$ is a future directed timelike unit normal.
\end{proof}

\section{Finiteness of total angular momentum and total center of mass}
We recall how the optimal isometric embedding equation is solved for asymptotically flat initial data set of order $1$ in \cite{Chen-Wang-Yau1}. Given $(\sigma, |H|, \alpha_H)$, we look for an isometric embedding $X=(X^0, X^1, X^2, X^3)$ into the Minkowski spacetime and a $T_0$ such that 
\begin{equation}
\begin{cases}
& \langle dX, dX\rangle=\sigma\\
&div_\sigma\left(\rho \nabla \tau - \nabla [ \sinh^{-1} (\frac{\rho \Delta \tau }{|H_0||H|})] - \alpha_{H_0} + \alpha_{H}\right)=0. 
\end{cases}
\end{equation} where $\tau=-\langle X, T_0\rangle$ and $\rho$ is given by \eqref{rho}.

We will indeed solve the equation for each coordinate sphere $\Sigma_r$ and expand the solution $X({r}), T_0({r})$ in power series of $r$.  We start with $X^i({r})=r\tilde{X}^i+ (X^i)^{(0)}+O(r^{-1})$. As we emphasize, this choice anchor the optimal isometric embedding. 

$X^0({r})=(X^0)^{(0)}+O(r^{-1})$ and $X^i({r})=r\tilde{X}^i+ (X^i)^{(0)}+O(r^{-1})$, for functions $(X^0)^{(0)}$ and $(X^i)^{(0)}$ of order $O(1)$. 

It turns out by Lemma 4 in \cite{Chen-Wang-Yau1}, $({X}^i)^{(0)}$ satisfies the linear PDE system:

\begin{equation}\label{linearized_isom} 2\sum_{i=1}^3 d \tilde{X}^i \cdot d({X}^i)^{(0)}=\sigma^{(1)}.\end{equation} This can be solved by writing $({X}^i)^{(0)}=p^a \tilde \nabla_a \tilde{X}^i+ v \tilde{X}^i$ and solve
for a $(0, 1)$ tensor $p^a$ and a function $v$ on $S^2$. With $(X^i)^{(0)}$, we compute
the mean curvature of the isometric embedding $H_0$ whose norm is given by  
$|H_0| ={2}{r}^{-1}+r^{-2}h_0^{(-2)}+ O(r^{-3})$, where
\[h_0^{(-2)}=-\tilde{X}^i\widetilde{\Delta}({X}^i)^{(0)}-\tilde{\sigma}^{ab}\sigma_{ab}^{(1)}.\]  The integral, by \eqref{linearized_isom}, is 
\[\int_{S^2} h_0^{(-2)}=-\frac{1}{2}\int_{S^2}\tilde{\sigma}^{ab} \sigma_{ab}^{(1)}.\]

We then write $T_0({r})=(a^0, a^1, a^2, a^3)+O(r^{-1})$ and proceed to solve $(a^0, a^1, a^2, a^3)$ and $(X^0)^{(0)}$. Here $\tau=-\langle X, T_0\rangle= -r a_i \tilde{X}^i+O(1)$. Let $\alpha_{H_0}$ be the one form  { $\langle  \nabla^{\R^{3,1}}_{(\cdot) } \frac{ J_0 }{|H_0|}, \frac{H_0}{|H_0|} \rangle$} on the image of ${X}_r$
then
\begin{equation}\label{alpha_H_0}
(\alpha_{H_0})_a=\frac{1}{2}\tilde \nabla_a  \left(\widetilde \Delta  (X^0)^{(0)}+2 (X^0)^{(0)}\right) r^{-1}+ O(r^{-2}).   
\end{equation}

The leading $r^{(-3)}$ term of the optimal isometric embedding equation is the following elliptic equation on $S^2$

 \begin{equation}\label{linearized_optimal}\widetilde{div}(\rho^{(-2)}\widetilde{\nabla}\tau^{(1)})-\frac{1}{4} \widetilde{\Delta} (\rho^{(-2)}\widetilde{\Delta} \tau^{(1)})-\frac{1}{2}\widetilde{\Delta}(\widetilde{\Delta}+2)(X^0)^{(0)}+\widetilde{div}( \alpha_H^{(-1)})=0\end{equation} 
 where
\begin{equation}\label{rho-2} \rho^{(-2)}=\frac{h_0^{(-2)}-h^{(-2)}}{a^0}\end{equation} and \[\tau^{(1)} =-\sum_{i=1}^{3} a_{i} \tilde X^i.\] 

The unknowns in \eqref{linearized_optimal} are $(X^0)^{(0)}$ and $(a^0, a^i)$.
For $(X^0)^{(0)}$ to be solvable, the integral against  any element in the kernel of the operators $\widetilde{\Delta}+2$, or $\tilde{X}^k, k=1, 2, 3$ must be zero. We derive that $(a^0, a^1, a^2, a^3)$ must satisfies 

\[a^k \int_{S^2}  (h_0^{(-2)}-h^{(-2)}) dS^2=a^0\int_{S^2} \tilde{X}^k \widetilde{div}(\alpha_H^{(-1)})\]

Since  the integral of $\rho$ gives the quasi-local mass, the leading term $m$ is given by 

\[m=\frac{1}{8\pi}\int_{S^2} \rho^{(-2)}dS^2=\frac{1}{8\pi}\frac{1}{a^0} \int_{S^2}  (h_0^{(-2)}-h^{(-2)}) dS^2.\]

We have 

\[\frac{1}{8 \pi}\int_{S^2} (h_0^{(-2)}-h^{(-2)}) dS^2=ma^0=e\] is the total energy and 

\[\frac{1}{8 \pi} \int_{S^2} \tilde{X}^i \widetilde{div}(\alpha_H^{(-1)})  dS^2=m a^i=p^i\] is the total momentum.

\begin{proposition} \label{condition_finite}
The total center of mass and total angular momentum are finite if  
\begin{equation}\label{condition_converge}
\begin{split}
\int_{S^2} \tilde X^i \rho^{(-2)} dS^2 =0\text{ and } \int_{S^2}  \tilde X^i \left( \tilde{\epsilon}^{ab}\tilde{\nabla}_b(\alpha_H^{(-1)})_a \right) dS^2=0
\end{split}
\end{equation}
\end{proposition}

\begin{proof}
We exam the finiteness of total conserved quantities with respect to a Killing field of the form $K_{\alpha\gamma} \eta^{\gamma\beta} X^{\alpha}\frac{\partial}{\partial X^\beta}$ with $K_{\alpha\gamma}+K_{\gamma\alpha}=0$. 

In \eqref{qlcq2},  $\rho$ is of the order of $r^{-2}$, while $\langle K, T_0\rangle$ could be of the order of $r$. Thus the top order term $-\int_\Sigma \langle K, T_0\rangle \rho d\Sigma$
needs to vanish in order to have a well-defined limit. 

Likewise, \[\begin{split}&\int_\Sigma K^\top\cdot \left(  -\rho {\nabla \tau }+ \nabla[ \sinh^{-1} (\frac{\rho \Delta \tau }{|H||H_0|})]-\alpha_H  + \alpha_{H_0} \right)  d \Sigma \\
&=
\int_\Sigma \langle K, \frac{\partial X}{\partial u^a} \rangle \sigma^{ab} j_b d\Sigma\end{split}\] where
 $j_b= \left(  -\rho{\nabla \tau }+ \nabla[ \sinh^{-1} (\frac{\rho\Delta \tau }{|H||H_0|})]-\alpha_H + \alpha_{H_0} \right)_b$ is of the order of $r^{-1}$ and $\langle K, \frac{\partial X}{\partial u^a} \rangle$ is of the order of $r^2$, thus the top order term needs to vanish as well.

Therefore, the $O({r})$ term of the conserved quantity is 
\begin{equation} \begin{split}&(E(\Sigma, X, T_0, K))^{(1)}=\frac{1}{8\pi} \int_{S^2}
-\langle K, T_0\rangle^{(1)} \rho^{(-2)} +\langle K, \frac{\partial X}{\partial u^a} \rangle^{(2)} \widetilde{\sigma}^{ab}j_b^{(-1)}  dS^2.\end{split}\end{equation}

where \begin{equation}\begin{split}j_b^{(-1)}=
 -\rho^{(-2)}\widetilde{\nabla}_b\tau^{(1)}+\frac{1}{4} \widetilde{\nabla}_b (\rho^{(-2)}\widetilde{\Delta} \tau^{(1)})-(\alpha_H)_b^{(-1)}+\frac{1}{2}\widetilde{\nabla}_b(\widetilde{\Delta}+2)(X^0)^{(0)}
\end{split}\end{equation}  satisfies the leading term of the optimal isometric embedding equation $\widetilde{\nabla}^b j_b^{(-1)}=0$.

In terms of coordinates, the leading term is $8\pi$ times 
\begin{equation}\label{leading}\int_{S^2} \left(-K_{\alpha\gamma} \tilde{X}^\alpha (T_0^\gamma)^{(0)} \rho^{(-2)}\right) dS^2+\int_{S^2}  \left(K_{\alpha\gamma}\tilde{X}^\alpha (\tilde \nabla_a  \tilde{X}^\gamma )\widetilde{\sigma}^{ab} j^{(-1)}_b\right) dS^2\end{equation} where we take $\tilde{X}^{0}=0$. 

Plugging $ (T_0^\gamma)^{(0)}=a^\gamma$ and expanding indexes, the first summand in \eqref{leading} is 
\[ -\int_{S^2} K_{\alpha\gamma} \tilde{X}^\alpha (T_0^\gamma)^{(0)} \rho^{(-2)} dS^2 =- \int_{S^2} \left(\sum_i(K_{i0} \tilde{X}^i a^0) + \sum_{i<j} K_{ij} (\tilde{X}^i a^j-\tilde{X}^j a^i)\right)\rho^{(-2)}dS^2. \]

We simplify the second summand in \eqref{leading} in the following. 
Using $\widetilde{\nabla}^b j_b^{(-1)}=0$ and integrating by parts, the second summand can be written as 
\[\frac{1}{2}\int_{S^2}K_{\alpha\gamma}[\tilde{X}^\alpha (\tilde \nabla_a  \tilde{X}^\gamma)-\tilde{X}^\gamma (\tilde \nabla_a  \tilde{X}^\alpha )] \widetilde{\sigma}^{ab} j_b^{(-1)} dS^2 .\]

We check that \[\tilde \nabla^a[\tilde{X}^\alpha (\tilde \nabla_a  \tilde{X}^\gamma)-\tilde{X}^\gamma (\tilde \nabla_a  \tilde{X}^\alpha ) ] =0.\] Therefore, we can throw away the two gradient terms in $j^{(-1)}_b$. On the other hand, we compute 
\[\frac{1}{2} K_{\alpha\gamma}[\tilde{X}^\alpha (\tilde \nabla_a  \tilde{X}^\gamma)-\tilde{X}^\gamma (\tilde \nabla_a  \tilde{X}^\alpha )  ]=\sum_{i<j} K_{ij}[\tilde{X}^i (\tilde \nabla_a  \tilde{X}^j)-\tilde{X}^j (\tilde \nabla_a  \tilde{X}^i )].\]

Therefore
\[
\begin{split}
&\int_{S^2}\left(K_{\alpha\gamma}\tilde{X}^\alpha (\tilde \nabla_a \tilde{X}^\gamma )\widetilde{\sigma}^{ab} j^{(-1)}_b\right) dS^2\\
=&\int_{S^2} \sum_{i<j} K_{ij} [\tilde{X}^i  (\tilde \nabla_a \tilde{X}^j)-\tilde{X}^j(\tilde \nabla_a \tilde{X}^i)]\widetilde{\sigma}^{ab}(-\rho^{(-2)}\widetilde{\nabla}_b\tau^{(1)}-(\alpha_H)_b^{(-1)}) dS^2.
\end{split}
\]

We compute 

\[(\tilde{X}^i  \tilde \nabla_a \tilde{X}^j-\tilde{X}^j \tilde \nabla_a  \tilde{X}^i)\widetilde{\sigma}^{ab}\tilde \nabla_b \tilde{X}^k=\tilde{X}^i \delta^{jk}-\tilde{X}^j \delta^{ik}\] and
\[\tilde{X}^i  (\tilde \nabla_a \tilde{X}^j)-\tilde{X}^j(\tilde \nabla_a \tilde{X}^i)= \epsilon_a^{\,\,b} \epsilon^{ij}_{\,\,\,\,k}\tilde \nabla_b \tilde{X}^k.
\] 

As a result, 
\[
\begin{split}
   &\int_{S^2}\left(K_{\alpha\gamma}\tilde{X}^\alpha (\tilde \nabla_a \tilde{X}^\gamma )\widetilde{\sigma}^{ab} j^{(-1)}_b\right) dS^2 \\
= & \int_{S^2} \left(\sum_{i<j} K_{ij}(\tilde{X}^i a^{j}-\tilde{X}^j a^{i})  \rho^{(-2)}- \sum_{i<j} K_{ij} [\tilde{X}^i  (\tilde \nabla_a \tilde{X}^j)-\tilde{X}^j(\tilde \nabla_a \tilde{X}^i)]\widetilde{\sigma}^{ab}(\alpha_H)_b^{(-1)} \right)dS^2. \\
= & \int_{S^2} \left(\sum_{i<j} K_{ij}(\tilde{X}^i a^{j}-\tilde{X}^j a^{i})  \rho^{(-2)} + \epsilon^{ij}_{\,\,\,\,k} K_{ij} \tilde X^k \tilde{\epsilon}^{ab}\tilde{\nabla}_b(\alpha_H^{(-1)})_a \right) dS^2.
\end{split}
\]
The condition \eqref{condition_converge} implies both summands in \eqref{leading} vanish.

\end{proof}

In the rest of the section, we prove the finiteness of total conserved quantities under the vacuum constraint equation. 
Recall that an initial data set $(M,g,k)$ satisfies the vacuum constraint equation if 

\begin{equation}\label{constraint}
\begin{split}
R(g) + (tr k)^2-|k|^2 &=0\\
\nabla_g^i k_{ij} - \partial_j tr_g k &=0
\end{split}
\end{equation}
where $R(g)$ is the scalar curvature of $g$ and $\nabla_g$ is the covariant derivative with respect to $g$ .

We first need a lemma regarding initial data sets satisfying the vacuum constraint equation. 

\begin{lemma}\label{physical_finite}
Let $(M,g,k)$ be a vacuum asymptotically flat initial data of order 1. Consider the coordinate spheres $\Sigma_r$ and let $\hat h$ be the mean curvature of the coordinate spheres in $M$. Assume that 
\[ \hat h = \frac{2}{r} + \frac{\hat h^{(-2)}}{r^2} +O(r^{-3})  \]
then 
\[ \int_{S^2} \tilde X^i \hat h^{(-2)} dS^2 =0.   \]
\end{lemma}
\begin{proof}
By the second variation formula of the area of the coordinate spheres, we have
\begin{equation} \label{second_variation_area}
\partial_r \hat h = -f(Rc(\nu,\nu) +|A|^2) - \Delta f \end{equation}
where $Rc$ is the Ricci curvature of the induced metric on $M$, $\nu$ is the outward unit normal of $\Sigma_r$, $f$ is the lapse function of the coordinate spheres and $A$ is the second fundamental form of the coordinate spheres in $M$. The Gauss equation of the surfaces $\Sigma_r$ in M is
\begin{equation} \label{gauss}
 K= \frac{R(g)}{2}- Rc(\nu,\nu)  + \frac{1}{2}(\hat h^2-|A|^2).  \end{equation}
Combining equation \eqref{second_variation_area} and \eqref{gauss}, we have 
\begin{equation}\label{variation_mean_curvature} 
 \partial_r \hat h = f( K  -  \frac{R(g)}{2}  - \frac{1}{2}(\hat h^2-|A|^2)  -|A|^2) - \Delta f .
\end{equation}
We have the following aymptotic expansion for $K$, $f$, $\hat h$:
\[
\begin{split}
K & = \frac{1}{r^2} + \frac{K^{(-3)}}{r^3} +O(r^{-4}),  \\
\hat h & = \frac{2}{r} + \frac{{\hat h}^{(-2)}}{r^2} +O(r^{-3}),\\
f & =  1+\frac{f^{(-1)}}{r}.+ O(r^{-2})
\end{split}
\]
In addition, the second fundamental form satisfies
\begin{equation}
|A|^2 = \frac{\hat h^2}{2} + O(r^{-4}).
\end{equation}
and $R(g) = O(r^{-4})$ by the vacuum constraint equation. Hence, from equation (\ref{variation_mean_curvature}), it follows that 
\begin{equation}   \label{variation_mean_curvature_2}
\partial_r \hat h = f(K - \frac{3 \hat h^2}{4}) - \Delta f + O(r^{-4})
\end{equation}
and
\[  \hat h^{(-2)} = K^{(-3)}  - \widetilde \Delta f^{(-1)} -2f^{(-1)}.  \]
Integration by parts, we derive
\[  \int_{S^2} \tilde X^i {\hat h}^{(-2)} dS^2 = \int_{S^2} \tilde X^i( K^{(-3)}  - \widetilde \Delta f^{(-1)} -2f^{(-1)}) dS^2 = \int_{S^2} \tilde X^i K^{(-3)} dS^2.  \]
To compute the Gauss curvature $K$ of the coordinate sphere, recall that if $\gamma^c_{ab}$ are the Christoffel symbols of the induced metric, then
\[2K= \sigma^{bd}(\partial_a \gamma^a_{bd} - \partial_d \gamma^a_{ba} + \gamma^a_{af}\gamma^f_{db} - \gamma^a_{df} \gamma^f_{ab}).  \]
Moreover,  we have the following asymptotics for the Christoffel symbols  
 \[ \gamma_{ab}^c = \tilde \gamma^c_{ab} + \frac{\tilde \sigma^{cd}}{2r}(\tilde \nabla_b  \sigma_{ad}^{(1)} +\tilde \nabla_a  \sigma_{bd}^{(1)} -\tilde \nabla_d  \sigma_{ab}^{(1)}  ) +O(r^{-2}), \] where $\tilde \gamma^c_{ab}$ are the Christoffel symbols of $\tilde{\sigma}_{ab}$.
$2K$ has the following expansion: 
\[ 2K = \frac{2}{r^2} + \frac{1}{r^3} \left( -\tilde \sigma^{ab}\sigma^{(1)}_{ab}+ \tilde \sigma^{bd}[\tilde \nabla_a (\gamma^{(-1)})_{bd}^a - \tilde \nabla_d (\gamma^{(-1)})_{ba}^a) ]\right) + O(r^{-4}) \]
where 
\[  (\gamma^{(-1)})_{ab}^c = \frac{1}{2} \tilde{\sigma}^{cd} (\tilde \nabla_b  \sigma_{ad}^{(1)} +\tilde \nabla_a  \sigma_{bd}^{(1)} -\tilde \nabla_d  \sigma_{ab}^{(1)}  ).\]
As a result,
\[
\begin{split} 
  & K^{(-3)} \\ 
=& -\tilde \sigma^{ab}\sigma^{(1)}_{ab} + \frac{\tilde \sigma^{ae}\tilde \sigma^{bd}}{2} (\tilde \nabla_a \tilde \nabla_b \sigma^{(1)}_{de}+\tilde \nabla_a \tilde \nabla_d \sigma^{(1)}_{be}-\tilde \nabla_a \tilde \nabla_e \sigma^{(1)}_{bd}-\tilde \nabla_d \tilde \nabla_b \sigma^{(1)}_{ae}-\tilde \nabla_d \tilde \nabla_a \sigma^{(1)}_{be}+\tilde \nabla_d \tilde \nabla_e \sigma^{(1)}_{ab})\\
  =& -\tilde \sigma^{ab}\sigma^{(1)}_{ab} + \tilde \nabla^a \tilde \nabla ^b \sigma^{(1)}_{ab}  - \widetilde \Delta ( \tilde \sigma^{ab}\sigma^{(1)}_{ab}).
\end{split}
\]
Integration by parts, we obtain
\[  \int_{S^2} \tilde X^i K^{(-3)}dS^2 = \int_{S^2} \tilde X^i [ -\tilde \sigma^{ab}\sigma^{(1)}_{ab} + \tilde \nabla^a \tilde \nabla ^b \sigma^{(1)}_{ab}  - \widetilde \Delta ( \tilde \sigma^{ab}\sigma^{(1)}_{ab}) ] dS^2 =0. \]
\end{proof}

We also need the following result regarding isometric embeddings into $\R^3$. 

\begin{lemma} \label{reference_finite}
Suppose $\Sigma_r$ is a family of surfaces with induced metric $\sigma = r^2\tilde \sigma + r \sigma^{(1)} +O(1)$ and let 
\[ \hat h_0 = \frac{2}{r} + \frac{\hat h_0^{(-2)}}{r^2} +O(r^{-3}) \]
be the mean curvature of the isometric embedding of $\Sigma_r$ into $\R^3$, then we have
\[ \int_{S^2} \tilde X^i \hat h_0^{(-2)} dS^2 =0. \]
\end{lemma}
\begin{proof}

We derive that
\[  \hat{h}_0^{(-2)}=-\sum_i \tilde{X}^i\widetilde{\Delta}Y^i-\tilde{\sigma}^{ab}\sigma_{ab}^{(1)},\]  
where $Y^i, i=1, 2, 3$ are functions on $S^2$ that solve  the linearized isometric embedding equation: 
\begin{equation} \sum_{i=1}^3  \tilde \nabla _a \tilde{X}^i  \tilde \nabla_b Y^i + \tilde \nabla _b \tilde{X}^i  \tilde \nabla_a Y^i =\sigma_{ab}^{(1)}.\end{equation} 
Hence,
\[
 \int_{S^2} \tilde X^i \hat h_0^{(-2)} dS^2 
=  -  \int_{S^2} (\tilde X^i \tilde{\sigma}^{ab}\sigma_{ab}^{(1)} + \tilde X^i \sum_j\tilde{X}^j\widetilde{\Delta}Y^j ) dS^2.
\]
To compute this, we use the ansatz in Nirenberg's  paper \cite{n} for solving the isometric embedding of surface with positive Gauss curvature into $\R^3$. 
Let
\begin{equation} \tilde \nabla_a  Y^j= P_a \tilde X + (\frac{\sigma^{(1)}_{ab}}{2} + F \epsilon_{ab}) \tilde \sigma^{bc} \tilde X_c  \end{equation}
where $ \epsilon_{ab}$ is the area form.  The one-form $P_a$ and the function $F$ are the new unknowns. 
Instead of the original isometric embedding equation, one looks at the compatibility condition as the new set of equations. 
 
We derive
\begin{equation} \label{Nirenberg_equ}
 \tilde \nabla_d \tilde \nabla_a Y = (\tilde \nabla_d P_a) X  + P_a \tilde X_d +  (\frac{\tilde \nabla _d\sigma^{(1)}_{ab}}{2} + F_d \epsilon_{ab}) \tilde \sigma^{bc} \tilde X_c -  (\frac{\sigma^{(1)}_{ad}}{2} + F \epsilon_{ad}) \tilde X .
\end{equation}
Taking the trace of equation (\ref{Nirenberg_equ}), we have 
\[\sum_i   \tilde X^ i \tilde \Delta Y^i  = \tilde \nabla^a P_a - \frac{1}{2} \tilde \sigma^{ab}\sigma^{(1)}_{ab}.\]

From \eqref{dth_0} and the above equation, we obtain

\begin{equation}\label{integral_dth_0}\int \tilde{X}^i h_0^{(-2)}dS^2 =\frac{-1}{2}\int \tilde{X}^i\tilde{\sigma}^{ab}\sigma^{(1)}_{ab}dS^2 -\int \tilde{X}^i\tilde{\nabla}^a P_a dS^2.\end{equation}
To compute $\tilde \nabla^a P_a $, we use the compatibility condition, $\epsilon^{ad}\tilde \nabla_d \tilde \nabla_a Y =0 $, to obtain the equation for $P_a$ and $F$. Again, using equation (\ref{Nirenberg_equ}), we can express the compatibility condition into the following equations:
\[
\begin{split}
\epsilon^{ad}( \tilde \nabla_d P_a  - \frac{\sigma^{(1)}_{ad}}{2} - F \epsilon_{ad})& =0,\\
\epsilon^{ad}(P_a \tilde \sigma_{cd} +\frac{1}{2}\tilde \nabla _d\sigma^{(1)}_{ac} +\tilde \nabla _d  F \epsilon_{ac}) & =0.
\end{split}
\]
As shown in \cite{n}, this is an elliptic system for $P_a$ and $F$. Indeed, one may express $P_a$ in terms of $F$ from the second equation then replace them in the first one to obtain an second order elliptic equation of $F$.  We can solve $ P_a$ from the second equation in terms of $F$ and $T_{ab}$. The equation can be written as 
\begin{equation} \label{equ_P}
\epsilon_{ac}P_b \tilde \sigma^{ab} +\frac{1}{2} \epsilon^{ad} \tilde\nabla _d\sigma^{(1)}_{ac}+ \tilde \nabla _c F   =0.
\end{equation}
Multiplying  the equation by $\epsilon^{ce}$, we have
\[ P^e = \frac{1}{2} \epsilon^{ce} \epsilon ^{ad} \tilde\nabla _d\sigma^{(1)}_{ac} + \epsilon^{ce} \tilde \nabla _c F.    \]
Integration by parts and simplifying yield
\[
\begin{split}
\int_{S^2} \tilde X^i \tilde \nabla ^e P_e  dS^2 &= \int_{S^2} \tilde X_i \tilde \nabla _e(\frac{1}{2} \epsilon^{ce} \epsilon ^{ad} \tilde\nabla _d\sigma^{(1)}_{ac} + \epsilon^{ce} \tilde \nabla _c F) dS^2 \\
& =  \frac{1}{2}\int_{S^2}\tilde \nabla _e \tilde\nabla _d \tilde X^i  \epsilon^{ce} \epsilon ^{ad} \sigma^{(1)}_{ac}dS^2\\
& =  -\frac{1}{2} \int_{S^2} \tilde X^i \tilde \sigma_{de}  \epsilon^{ce} \epsilon ^{ad}\sigma^{(1)}_{ac}dS^2\\
& = - \frac{1}{2} \int_{S^2} \tilde X^i \tilde \sigma^{ab}\sigma^{(1)}_{ab} dS^2.
\end{split}
\]
Plugging this into \eqref{integral_dth_0}, we obtain 
\[  \int_{S^2}  \tilde X^i  h_{0}^{(-2)} dS^2 =0.\]
\end{proof}

\begin{theorem} \label{finite}
For a vacuum asymptotically flat initial data set of order $1$ (see Definition \ref{order_one}), 
the total angular momentum and total center of mass are finite.
\end{theorem}
\begin{proof}
By Proposition \ref{condition_finite}, it suffices to show that 
\[
 \int_{S^2} \tilde X^i \rho^{(-2)} dS^2 =0\text{ and } \int_{S^2}  \tilde X^i \left( \tilde{\epsilon}^{ab}\tilde{\nabla}_b(\alpha_H^{(-1)})_a \right) dS^2=0.
\]
Recall that 
\[  \rho^{(-2)} = \frac{h_0^{(-2)}-h^{(-2)}}{a^0}.    \]
By Lemma 4 of \cite{Chen-Wang-Yau1}, 
\[ \hat h_0^{(-2)} = h_0^{(-2)}.  \]
By Lemma \ref{reference_finite}, 
\[  \int_{S^2} \tilde X^i h_0^{(-2)} d S^2= \int_{S^2} \tilde X^i \hat h_0^{(-2)}  d S^2 = 0.\]

For $\int_{S^2}\tilde X^i  h^{(-2)}$, we have
\[ h^{(-2)} = \hat h^{(-2)}  \]
since $k=O(r^{-2})$.
By Lemma \ref{physical_finite},
\[  \int_{S^2} \tilde X^i  h^{(-2)} d S^2= \int_{S^2} \tilde X^i \hat h^{(-2)} dS^2 =0. \]
To prove  $  \int_{S^2}  \tilde X^i \left( \tilde{\epsilon}^{ab}\tilde{\nabla}_b(\alpha_H^{(-1)})_a \right) dS^2=0, $
it suffices to show that 
\[   \int_{S^2}  \tilde X^i \left( \tilde{\epsilon}^{ab}\tilde{\nabla}_b(\pi_{ar }^{(-1)}) \right) dS^2=0\]
since $(\alpha_H^{(-1)})_a$ and $-\pi_{ar }^{(-1)}$ differs only by a gradient vector field. 

By the vacuum constraint equation, $\nabla_g^i \pi_{a i} = 0.$ In the appendix, we prove that 
\begin{equation}\label{eq:constraint1}
 \nabla^i_g \pi_{ ia} =  \frac{\pi_{ar }^{(-1)} - \tilde \nabla^b \pi^{(0)}_{ab}}{r^2} +O(r^{-3}). 
 \end{equation}
and thus
\[ \pi_{ar }^{(-1)} = \tilde \nabla^b \pi^{(0)}_{ab}. \]
Integration by parts, we derive
\[  
\begin{split}
 \int_{S^2}  \tilde X^i \left( \tilde{\epsilon}^{ab}\tilde{\nabla}_b(\pi_{ar }^{(-1)}) \right) dS^2= & \int_{S^2}  \tilde X^i( \tilde{\epsilon}^{ab}\tilde{\nabla}_b\tilde{\nabla}^c\pi^{(0)}_{ac} ) d S^2\\
 = & \int_{S^2}  (\tilde{\nabla}_b\tilde{\nabla}^c \tilde X^i) \tilde{\epsilon}^{ab}\pi^{(0)}_{ac}d S^2\\
= & -\int_{S^2}  \tilde{X}^i \tilde{\epsilon}^{ab}\pi^{(0)}_{ab}d S^2.
\end{split} 
\] The last expression is zero because $\pi_{ab}^{(0)}$ is symmetric. 
\end{proof}
In the rest of this section, we prove the following lemma. The result is similar to Theorem \ref{finite} but with weaker assumptions on the initial data.
This lemma will be useful in Section 10 to study evolution of center of mass and angular momentum.
\begin{lemma} \label{slow}
 Suppose $(M, g, k)$ is a asymptotically flat  vacuum initial data satisfying 
\begin{equation}
\begin{split}
g_{ij}=&\delta_{ij}+O(r^{-1})\\
k_{ij}=&O(r^{-2}).
\end{split}
\end{equation}
Let $\Sigma_{r}$ be the coordinate sphere and $(X(r),T_0(r))$ be a family of isometric embeddings and observers with the following expansion
\[  
\begin{split}
X^0({r}) & = O(1) \\
X^i ({r})& = r \tilde X^i + O(1) \\
T_0 ({r})& = O(1).
\end{split}
\] 
Then
\begin{equation}
\begin{split}
\int_{\Sigma_r}X^i(|H_0|-|H|) d \Sigma_r = & O(\ln r) \\
\int_{\Sigma_r} (X^i \nabla_a X^j -X^j \nabla_a X^i)\sigma^{ab}j_b d \Sigma_r =& O(\ln r).
\end{split}
\end{equation}
\end{lemma}
\begin{proof}
First we write
\[  \int_{\Sigma_r}X^i(|H_0|-|H|) d \Sigma_r = \int_{\Sigma_r}X^i(|H_0|-\frac{2}{r}) d \Sigma_r  -  \int_{\Sigma_r}X^i(|H|-\frac{2}{r}) d \Sigma_r  \]
For $\int_{\Sigma_r}X^i(|H_0|-\frac{2}{r}) d \Sigma_r$, following the same argument used in the proof of Lemma \ref{reference_finite} with $\sigma^{(1)}_{ab}$ replaced by $\frac{\sigma_{ab}-r^2\tilde \sigma_{ab}}{r}$, it is straightforward to derive
\[  \int_{\Sigma_r}X^i(|H_0|-\frac{2}{r}) d \Sigma_r = O(1).  \]
Therefore it suffices to show that 
\[  \partial_r \int_{\Sigma_r}X^i(|H|-\frac{2}{r}) d \Sigma_r = O(r^{-1}). \]
We compute 
\begin{equation}
\begin{split}\label{derivative_physical}
   & \partial_r \int_{\Sigma_r}(|H|-\frac{2}{r})X^i d \Sigma_r  \\
= & \int_{\Sigma_r} ( \partial_r |H| + \frac{2}{r^2} )X^i + \frac{3}{r}(|H|-\frac{2}{r})X^i d \Sigma_r  + O(r^{-1})\\
=  &  \int_{\Sigma_r} ( \partial_r \hat h - \frac{4}{r^2} + \frac{3 \hat h}{r})X^i d \Sigma_r   + O(r^{-1})
\end{split}
\end{equation}
where $\hat h$ is the mean curvature of $\Sigma_r$ in $M$ as in Lemma \ref{physical_finite}. Furthermore, recall equation \eqref{variation_mean_curvature_2}
\[
\partial_r \hat h = f( K  - \frac{3\hat h^2}{4}) - \Delta f  + O(r^{-4}) \]
where $f$ is the lapse of the coordinates spheres, $K$ is the Gauss curvature of the coordinate sphere. We have the expansions for $K$ and $f$:
\[  
K = \frac{1}{r^2} + O(r^{-3}),  f=1 +O(r^{-1}).
\]
Hence, from equation \eqref{variation_mean_curvature_2},  it follows that 
\begin{equation}\label{variation_mean_curvature_3} 
 \begin{split} & \partial_r \hat h  \\
=& (f-1)( K  - \frac{3\hat h^2}{4})+ ( K  - \frac{3\hat h^2}{4}) - \Delta (f-1)  + O(r^{-4})\\
= & \frac{-2(f-1)}{r^2}  - \Delta(f-1) +  (K - \frac{3\hat h^2}{4}  )  + O(r^{-4}).
\end{split}
\end{equation}

 Combining equation \eqref{derivative_physical} and \eqref{variation_mean_curvature_3}, we have
\[  
\begin{split}
   &\partial_r \int_{\Sigma_r}(|H|-\frac{2}{r})X^i d \Sigma_r  \\
=  & \int_{\Sigma_r} [ - \Delta (f-1)  - \frac{2(f-1)}{r^2} +K - \frac{3\hat h^2}{4}  - \frac{4}{r^2} + \frac{3\hat h}{r}]X^i d \Sigma_r  + O(r^{-1}) \\
=  & \int_{\Sigma_r} [ - \Delta (f-1)  - \frac{2(f-1)}{r^2} +(K -\frac{1}{r^2}) - \frac{3\hat h^2}{4}  - \frac{3}{r^2} + \frac{3\hat h}{r}]X^i d \Sigma_r  + O(r^{-1}).  
\end{split}
\]
Moreover, 
\[ \int_{\Sigma_r} [ \Delta (f-1)   +\frac{2(f-1)}{r^2} ]  X^i d \Sigma_r =   \int_{\Sigma_r} (f-1)  (\Delta + \frac{2}{r^2})X^i d \Sigma_r  = O(r^{-1})  \]
\[\int_{\Sigma_r} [- \frac{3\hat h^2}{4}  - \frac{3}{r^2} + \frac{3\hat h}{r}]X^i d \Sigma_r = \int_{\Sigma_r}\frac{-3}{4} (\hat h - \frac{2}{r})^2X^i d \Sigma_r = O(r^{-1}). \]
Finally, 
 \[  \int_{\Sigma_r}(K -\frac{1}{r^2}) X^i d \Sigma_r  =O(r^{-1})\]
can be verified using the same argument for computing $K^{(-3)}$ in Lemma \ref{physical_finite},  with 
$\sigma_{ab}^{(1)}$ replaced by $\frac{\sigma_{ab}- r^2\tilde \sigma_{ab}}{r}$. This proves
\[ \int_{\Sigma_r}X^i(|H_0|-|H|) d \Sigma_r =  O(\ln r). \]
To show that 
\[ \int_{\Sigma_r} (X^i \nabla_a X^j -X^j \nabla_a X^i)\sigma^{ab}j_b d \Sigma_r = O(\ln r), \]
we first recall 
\[j_b= \left(  -\rho{\nabla \tau }+ \nabla[ \sinh^{-1} (\frac{\rho\Delta \tau }{|H||H_0|})]-\alpha_H + \alpha_{H_0} \right)_b.\]
Since $ \nabla[ \sinh^{-1} (\frac{\rho\Delta \tau }{|H||H_0|})]$ is a gradient vector filed, we have
\[ 
\int_{\Sigma_r} (X^i \nabla_a X^j -X^j \nabla_a X^i)\sigma^{ab}\nabla_b [ \sinh^{-1} (\frac{\rho\Delta \tau }{|H||H_0|})] \Sigma_r = O(1).
\]
Similarly,  
\[
\int_{\Sigma_r} (X^i \nabla_a X^j -X^j \nabla_a X^i)\sigma^{ab}(\alpha_{H_0} )_b d \Sigma_r = O(1)
\]
since $\alpha_{H_0}$ is a gradient vector field up to $O(r^{-2})$ terms by equation \eqref{alpha_H_0}. From \cite{Wang-Yau3},
\[ (\alpha_H)_a = -k_{\nu a}+ \nabla_a \frac{tr_{\Sigma}k }{|H|}  \text{ and } k_{\nu a} = k_{ra} + O(r^{-2}). \]
As a result,
\[ 
\begin{split}
  & \int_{\Sigma_r} (X^i \nabla_a X^j -X^j \nabla_a X^i)\sigma^{ab}(\alpha_H)_b  d\Sigma_r \\
=& -\int_{\Sigma_r} (X^i \nabla_a X^j -X^j \nabla_a X^i)\sigma^{ab}k_{b \nu}d\Sigma_r+  O(1).
\end{split}
\]
$ \int_{\Sigma_r} (X^i \nabla_a X^j -X^j \nabla_a X^i)\sigma^{ab}k_{b \nu}d\Sigma_r$, up to $O(1)$ difference, is the same as 
ADM angular momentum integral. Using divergence theorem, the ADM angular momentum can be written as an integral on $M$, see equation (7) of  \cite{Chrusciel2}. 
Given the asymptotically flat assumption, it is easy to see that the integrand in equation (7) of \cite{Chrusciel2} is of order $O(r^{-3})$. It follows that  the 
ADM angular momentum integral is of order  of $O(\ln r)$.

Finally, the leading term of 
\[
\begin{split}
  & \int_{\Sigma_r} (X^i \nabla_a X^j -X^j \nabla_a X^i)\Sigma^{ab}\rho \tau_b d \Sigma_r,\\
= & -\int_{\Sigma_r} \frac{|H_0| - |H|}{a^0}\sum_k a^k  (X^i \nabla_a X^j -X^j \nabla_a X^i)\Sigma^{ab} \nabla_k X^k  d \Sigma_r + O(1) \\
=& -\int_{\Sigma_r} \frac{|H_0| - |H|}{a^0} (a^jX^i  -a^i X^j )  d \Sigma_r + O(1) \\
=&  O(\ln r).
\end{split}
\]
This finishes the proof of the lemma.
\end{proof}

\section{Invariance of angular momentum in the Kerr spacetime}
In this section, we study the limit of the new quasi-local angular momentum at the infinity of a spacelike hypersurface in the Kerr spacetime. 
Fix the Boyer--Lindquist coordinate $\{  t,r, \theta, \phi \}$ for the Kerr spacetime. The metric  is (see page 313 of \cite{Wald})
\[\begin{split} 
&-(1-\frac{2mr}{r^2+a^2\cos^2 \theta})dt^2 +(\frac{r^2+a^2\cos^2 \theta}{r^2-2mr+a^2})dr^2+(r^2+a^2\cos^2 \theta)d\theta^2\\
&+(r^2+a^2+\frac{2mra^2\sin^2\theta}{r^2+a^2\cos^2 \theta}) \sin^2\theta d\phi^2 - (\frac{4mar \sin^2 \theta}{r^2+a^2\cos^2 \theta}) dt d\phi.
\end{split}\]

The $t=0$ slice corresponds to a maximal slice. We consider a family of surface $\widehat \Sigma_R$ for $R>>1$ on this slice that is defined by $r=R$ and denote the induced metric on $\widehat \Sigma_R$ by $g_{ab}$. 

Suppose a spacelike hypersurface in the Kerr spacetime is defined by $t= f(r,u^a)$ for $r>>1$ where $f=o(r)$. Consider the family of the 2-surfaces $\Sigma_{R,f}$ defined by $r=R$ and $t=f(R,u^a)$ on this hypersurface. For $R$ large, the isometric embedding of the induced metric on $\Sigma_{R, f}$ into $\R^{3,1}$ with time function $\bar f_R(u^a)=\sqrt{1-\frac{2m}{R}} f(R,u^a)$ over a fixed $\R^3$ is an approximate solution to the optimal embedding equation which can be perturbed to a local minimum of the quasi-local energy. Denote the image of isometric embedding by $ \Sigma_{R, \bar f}$ and its projection to the fixed $\R^3$ by $\widehat \Sigma_{R, \bar f}$. 
We will compute the quasi-local angular momentum using this isometric embedding and show that the result is independent of the defining equation $f$ of the hypersurface.

First, we need to compare the geometry of the surface $\Sigma_{R,f}$ in the Kerr spacetime with that of $\widehat \Sigma_{R, \bar f}$ in $\R^{3,1}$.
The tangent bundle of the $\Sigma_{R,f}$ is spanned by
\[
 \frac{\partial}{\partial w^a} =  \frac{\partial}{\partial u^a}+  \partial_{a} f  \frac{\partial}{\partial t}.
\]
Using the isometric embedding, we push forward the tangent vectors $ \frac{\partial}{\partial w^a}$ to $\Sigma_{R, \bar f}$. Denote these tangent vectors by $ \frac{\partial}{\partial w^a}$ as well. Using the isometric embedding, a function $\tau $ on $\Sigma_{R,f}$ is also a function on   $\Sigma_{R, \bar f}$. We use $\partial_a \tau$ to denote the partial derivative of $\tau$ with respect to $ \frac{\partial}{\partial w^a}$. Also, we fix the standard coordinate $\{ s, \rho, v^a \}$ for Minkowski space and fix the constant vector $T_0= \frac{\partial}{\partial s}$. 
\begin{lemma}
Suppose $f=o(r)$. Let $\Sigma_{R,f}$ be the above surface in the Kerr spacetime determined by $t=f(R,u^a)$ and $r=R$. Let $\Sigma_{R,\bar f}$ be image of  the isometric embedding of $\Sigma_{R,f}$ into $\R^{3,1}$  with time function $\bar f$.  Let $H_{f}$ and   $H_{\bar f}$ be the mean curvature vectors of the surfaces in the Kerr spacetime and $\R^{3,1}$ respectively. Let $\{ \breve e_3 , \breve e_4 \}$  and $\{\bar  e_3, \bar e_4 \}$ be the canonical frame of the normal bundle of the surface in $\R^{3,1}$ and the Kerr spacetime determined by the isometric embedding (see  \S 6.2 of \cite{Wang-Yau2}). Finally, let $\alpha_{\breve e_3}$ and $\alpha_{\bar e_3}$ be the corresponding connection one form of normal bundle. Then, we have

\begin{align*}
-\langle H_f,  \bar e_3 \rangle & =\frac{(2 - |\hat \nabla \bar f|^2)\sqrt{1-\frac{2m}{R}}}{R(1-|\hat \nabla \bar f|^2)}-\frac{m|\hat \nabla \bar f |^2 }{R^2(1-|\hat \nabla \bar f|^2 )\sqrt{1-\frac{2m}{R}}} + O(R^{-3})\\
-\langle H_{\bar f},  \breve e_3 \rangle & =\frac{(2 - |\hat \nabla \bar f|^2)}{R(1-|\hat \nabla \bar f|^2)} + O(R^{-3})\\
\end{align*}
and
\begin{align*}
(\alpha_{\bar e_3})_a & = \frac{\partial_a \bar f}{R\sqrt{1-|\hat \nabla \bar f|^2 }} [ -\frac{m}{R\sqrt{1-\frac{2m}{R}}}+ \sqrt{1-\frac{2m}{R}} ]+ \langle \nabla^{Kerr}_a \partial_r, \partial_t \rangle+ o(R^{-2}) \\
(\alpha_{\breve e_3})_a & = \frac{\partial_a \bar f}{R\sqrt{1-|\hat \nabla \bar f|^2 }} + o(R^{-2})
\end{align*}
where $\hat \nabla$ denote the derivative with respect to the metric on $\widehat \Sigma_R$.
\end{lemma}
\begin{proof}
When $a=0$, the above formulas are exact without any error. We first verify these formiulas for $a=0$. For $a=0$, the cylinder $r=R$ is isometric to the cylinder $\rho=R$  in the Minkowski space by the identification $s=\sqrt{1-\frac{2m}{R}}t$ and $v^a=u^a$. Hence, the canonical gauge $\bar e_4$ is simply the unit normal of $\Sigma_{R,f}$ in the cylinder  $r=R$. 
It is  straightforward to verify that the second fundamental of the cylinder is 
\[-\frac{m(1-\frac{2m}{R})}{R^2\sqrt{1-\frac{2m}{R}}} dt^2 + \frac{\sqrt{1-\frac{2m}{R}}}{R}\sigma_{ab} du^a du^b  \]
and the unit normal vector $\bar e_4$ is 
\[ \bar e_4 = \frac{1}{\sqrt{1- |\hat \nabla \bar f|^2}} (\frac{1}{\sqrt{1-\frac{2m}{R}}} \frac{\partial}{\partial t} +  \bar f^a  \frac{\partial}{\partial u^a}).  \]
 It follows that 
\begin{align*}
-\langle H_f,  \bar e_3 \rangle & =\frac{(2 - |\hat \nabla \bar f|^2)\sqrt{1-\frac{2m}{R}}}{R(1-|\hat \nabla \bar f|^2)}-\frac{m|\hat \nabla \bar f |^2 }{R^2(1-|\hat \nabla \bar f|^2 )\sqrt{1-\frac{2m}{R}}}\\
(\alpha_{\bar e_3})_a & = \frac{\partial_a \bar f}{R\sqrt{1-|\hat \nabla \bar f|^2 }} [ -\frac{m}{R\sqrt{1-\frac{2m}{R}}}+ \sqrt{1-\frac{2m}{R}}].
\end{align*}
Moreover, the corresponding quantities in $\R^{3,1}$ are obtained by setting $m=0$.

When $a$ is non-zero, we have the corresponding error. For example, the metric is changed by $O(1)$. As a result, $\langle H_{\bar f},  \breve e_3 \rangle$ is changed by $O(R^{-3})$. Similarly, $|H_f|$ is changed by $O(R^{-3})$ and the frame $\bar e_3$ is changed by an angle of $O(R^{-2})$. 
\end{proof}
\begin{corollary} 
For the optimal embedding equation on $\Sigma_{R,f}$, $\tau=\bar f$ is an approximate solution with error $o(R^{-4})$. In particular, if we define 
$j_a$ using equation (\ref{optimal_embedding_vector}) with $\tau=\bar f$, then
\[ j_a =O(R^{-2}),   \qquad   \nabla^a j_a = o(R^{-4}).\]
\end{corollary}
Applying the iteration process for solving the optimal embedding equation  at spatial infinity \cite{Chen-Wang-Yau1}, we have
\begin{corollary} 
For the optimal embedding equation on $\Sigma_{R,f}$, there exists a solution $f_{min}$ of the form
\[  (f_{min})_R(u^a)=\bar f_R(u^a) + o(1)\tilde X^i+o(R^{-1}). \]
In particular, define $j'$ as in equation (\ref{optimal_embedding_vector}) with $\tau=f_{min}$, then
\[  j_a - j_a' =o(R^{-2}),   \qquad   \nabla^a j'_a = 0.\]
\end{corollary}

Let $X_R$ and $X'_R$ be the isometric embedding into $\R^{3,1}$ with time function $\bar f_R$ and $(f_{min})_R$. 
For the angular momentum, we use the Killing vector fields  $L_{ij}=X^i \frac{\partial}{\partial X^j}-X^j \frac{\partial}{\partial X^i} $.  We have the following invariance of total angular momentum result:
\begin{theorem}
For any function $f =o(r)$, we have
\[  \lim_{R \to \infty}E(\Sigma_{R,f},X'_R,T_0,L_{ij})=\lim_{R \to \infty}E( \widehat \Sigma_{R},X_0,T_0,L_{ij}).  \]
\end{theorem}
\begin{proof}
Let $L=L_{ij}$ and denote $\bar f_R$ and $(f_{min})_R$ by $\bar f$ and $f_{min}$.  Using the projections from $\Sigma_{R,f}$ to $\widehat \Sigma_{R}$ and from $\Sigma_{R,\bar f}$ to $\widehat \Sigma_{R, \bar f}$,
we can view functions on $\Sigma_{R,f}$ and $\Sigma_{R, \bar f}$ as functions on $\widehat \Sigma_{R}$ and  $\widehat \Sigma_{R, \bar f}$. Furthermore, we can also push forward the tangent vectors $\frac{\partial}{\partial w^a}$ to  $\widehat \Sigma_{R}$ and  $\widehat \Sigma_{R, \bar f}$ and still use $\partial_a$ to denote the partial derivative of functions. Finally, let $\mathcal C_1$ and $\mathcal C_2$ be the cylinder over $\widehat \Sigma_{R}$ and  $\widehat \Sigma_{R, \bar f}$.

We first compute the conserved quantity for $\Sigma_{R,f}$ with respect to the isometric embedding $X_R$. We have
\begin{proposition}
For $f=o(r)$, we have
\[  \lim_{R \to \infty}E(\Sigma_{R,f},X_R,T_0,L)=\lim_{R \to \infty}E( \widehat \Sigma_{R},X_0,T_0,L).  \]
\end{proposition}
\begin{proof}
The quasi-local angular momentum is the difference between the reference term and the physical term. Denote the reference term by $E_1$ and physical term by $E_2$. It is clear that 
\[
E_1( \widehat \Sigma_{R},X_0,T_0,L) = 0. 
\]
It suffices to show that 
\begin{align}
\label{claim1}\lim_{R \to \infty}E_1( \Sigma_{R,f},X_R,T_0,L) =& 0 \\
\label{claim2}\lim_{R \to \infty}E_2( \Sigma_{R,f},X_R,T_0,L) =&\lim_{R \to \infty} E_2( \widehat \Sigma_{R,f},X_0,T_0,L).
\end{align}
The induced metric on $\Sigma_{R,f}$ is 
\[ \sigma_{ab}=g_{ab} + f_a g_{tb}+ f_b g_{ta}+f_af_b g_{tt}= g_{ab} - \bar f_a \bar f_b +O(1). \]Hence, the induced metric on $\widehat \Sigma_{R,\bar f}$ is $ g_{ab}+O(1)$. As result, while the Killing vector field $L$ is not tangent to $\mathcal C_2$, we have
\[ \langle  \breve e_3, L\rangle = O(R^{-1}) . \]
Let 
\[ L^{\mathcal C_2}=L-\langle  \breve e_3, L\rangle \breve e_3\]
be the tangential component of $L$ to $\mathcal C_2$. Furthermore, let $L^\top$ be the tangential component of $L$ to $\Sigma_{R,\bar f}$
\[ L^{\top}=L-\langle  \breve e_3, L\rangle \breve e_3 + \langle  \breve e_4, L\rangle \breve e_4. \]

Consider the vector field $Z_2= \pi_2((L)^{\mathcal C_2} , \cdot )$ on $\mathcal C_2$ where $\pi_2$ 
is the conjugate momentum. We claim that 
\[ div Z_2= O(R^{-3}).\]
The conjugate momentum is divergence free since the spacetime is vacuum. Moreover, $L$ is Killing and thus
\[ (\mathcal L_{ L^{\mathcal C_2}} g)_{ij}= (\mathcal L_{\langle  \breve e_3, L\rangle \breve e_3} g)_{ij}=O(R^{-2}). \]
Also, it is clear that 
\[ \pi_2 = O(R^{-1}). \]
Hence
\begin{equation}\label{error1}
div Z_2= O(R^{-3}).    
\end{equation}

Using the divergence theorem for the vector field $V$, we have
\[ \int_{\widehat \Sigma_{R,\bar f}} \pi_2( L^{\mathcal C_2}, \hat n)  d \widehat \Sigma_{R, \bar f}  = \int_{ \Sigma_{R,\bar f}} \pi_2( L^{\mathcal C_2},  n) d  \Sigma_{R,\bar f} -\int_{\Omega_2} div Z_2  \]
where  $\Omega_2$ is the portion of $\mathcal C_2$ between $\widehat \Sigma_{\bar f}$ and $\Sigma_{\bar f}$.  Hence,
\[ \int_{\widehat \Sigma_{R,\bar f}} \pi_2(L^{\mathcal C_2}, \hat n) d \widehat \Sigma_{R, \bar f}   = \int_{ \Sigma_{R,\bar f}} \pi_2(L^{\mathcal C_2},  n) d  \Sigma_{R,\bar f} +o(1) . \]
However, $\widehat \Sigma_{R,\bar f}$ lies on the hypersurface $s=0$ in $\R^{3,1}$. As a result,
we have
\[   \int_{\widehat \Sigma_{R, \bar f }} \pi_2(L^{\mathcal C_2}, \hat n)  d \widehat \Sigma_{R, \bar f} =0.\]
By definition, we have
\[
\begin{split} 
\int_{\Sigma_{R,\bar f}} \pi_2(L^{\mathcal C_2}, n) d \Sigma_{R,\bar f} =& \int_{\Sigma_{R,\bar f}}(\langle L ,\breve e_4 \rangle \langle H_{\bar f}, \breve e_3 \rangle - \langle \nabla_{L^\top} \breve e_3,\breve e_4  \rangle ) d\Sigma_{R,\bar f}  \\
=& E_1( \Sigma_{R,f},X_R,T_0,L).
\end{split}
\]
This proves equation \eqref{claim1}.

For equation \eqref{claim2}, let $\pi_1$ be the conjugate momentum of $\mathcal C_1$. Identifying the coordinate $(t,r, \theta, \phi)$ in the Kerr spacetime and the spherical coordinate  $(s,\rho,u^a)$ in the Minkowski spacetime gives a diffeomorphism between exterior regions. We use this diffeomorphism to pull back the Killing vector field $L$ to Kerr spacetime and denote the pull back by 
$L'$.  $L'$ is an asymptotic Killing vector field. Consider the vector field $Z_1=\pi(L',\cdot)$ on $\mathcal C_1$.

 Let $\Omega_1$ be the portion of $\mathcal C_1$ bounded between $ \Sigma_{R,f}$ and $\hat \Sigma_R$. We have
\[ \int_{ \Sigma_{R,f}} Z_1 \cdot n  \, d \Sigma_{R,f}  =  \int_{\widehat \Sigma_R} Z_1 \cdot n\, d \widehat \Sigma_R + \int_{\Omega_1} div Z_1. \]
However,
\[ 
\begin{split}
 div Z_1 = & \pi^{ij}_1 \langle \nabla_{\partial_i} L', \partial_j \rangle  =O(R^{-3}).
\end{split}\]
Hence,
\[  \lim_{R \to \infty }\int_{ \Sigma_{R,f}} \pi_1( L' , n)  \, d \Sigma_{R,f}=  \lim_{R \to \infty }  \int_{\widehat \Sigma_R}  \pi_1( L' , \hat n) \, d\widehat \Sigma_R \]
where $\hat n$ and $n$ are the unit normal of $\widehat \Sigma_{R}$ and $\Sigma_{R,f}$ in  $\mathcal C_1$.

For the term on the right hand side, we have
\[ \int_{\widehat \Sigma_{R}} \pi_1(L', \hat n) \, d\widehat \Sigma_R= E_2( \widehat \Sigma_{R},X_0,T_0,L). \]
To compare the left hand side with $E_2(\Sigma_{R,f},X_R,T_0,L)$, recall that 
if $L$ is decomposed as
\[ L= \langle L ,\breve e_3 \rangle \breve e_3 -  \langle L ,\breve e_4 \rangle \breve e_4 + L^\top,\]
then 
\[  E_2(  \Sigma_{R,f},X_R,T_0,L) = \int_{\Sigma_{R,f}} (\langle L, \breve e_4 \rangle \langle H_{f}, \bar e_4 \rangle - \langle \nabla_{L^T} \bar e_3, \bar e_4 \rangle)   \, d \Sigma_{R,f} . \]
where $\bar e_3$ and $\bar e_4$ is the canonical frame of the normal bundle of $\Sigma_{R,f}$ determined by
\[  \langle H_f, \bar e_4 \rangle = \langle H_{\bar f}, \breve e_4 \rangle.   \]
On the other hand,
\[   \int_{\Sigma_{R,f}}  \pi_1(L',  n) \,   \, d \Sigma_{R,f}=  \int_{\Sigma_{R,f}} (\langle L', e_4 \rangle \langle H_{f},  e_4 \rangle - \langle \nabla_{(L')^\top}  e_3, e_4 \rangle )  \, d \Sigma_{R,f}\]
where $e_3$ is the unit normal of $\mathcal C_1$ and $e_4$ is the unit normal of $\Sigma_{R,f}$ in $\mathcal C_1$.

One has to compare the  vector fields $L$ to $L'$. We know that the frame $\{ e_3, e_4 \}$ where $e_3$ is the unit normal of $\mathcal C_1$ and $e_4$ is the unit normal of $\Sigma_{R,f}$  in $\mathcal C_1$ is a good approximation of the canonical frame, $\{\bar e_3, \bar e_4 \}$ and we have
\[ e_3= \bar e_3 + O(R^{-2})  \]
\[ e_4= \bar e_4 + O(R^{-2}).  \]
Indeed, for $a=0$, the equalities hold without any error. And $a>0$ introduces an error of order $O(R^{-2})$.

We have the following lemma which compare $L'$ with $L$.
\begin{lemma} \label{null_compare}
For the above coefficient, we have
\begin{align}  
 \langle L', e_4  \rangle &=  \langle L, \breve e_4  \rangle + o(R^{-1})\\
  |(L')^\top - L^\top | & =O(R^{-1}). \end{align}
\end{lemma}
\begin{proof}
We will only use the fact that the spacetime is asymptotically Schwarzschild. It suffices to consider the case when $L=\frac{\partial}{\partial v^2}$ and $L' =\frac{\partial}{\partial u^2}$.

The tangent bundle of the $\Sigma_R$ in the Kerr spacetime is spanned by 
\[
\begin{split}
 \frac{\partial}{\partial w^1} =&  \frac{\partial}{\partial u^1}+  \partial_{1} f  \frac{\partial}{\partial t}  \\
 \frac{\partial}{\partial w^2} =&  \frac{\partial}{\partial u^2} +  \partial_2 f  \frac{\partial}{\partial t}. 
\end{split}
\]
Hence, 
\[ \langle L',  \frac{\partial}{\partial w^a}  \rangle = r^2 \tilde \sigma_{2 a} + O(1)  \]

In terms of the coordinate of $\R^{3,1}$,  $ \frac{\partial}{\partial w^a}$ can be written as 
\[
\begin{split}
 \frac{\partial}{\partial w^1} =& (1+ O(R^{-2})) \frac{\partial}{\partial v^1}+ O(R^{-2})\frac{\partial}{\partial v^2}+ O(1)  \frac{\partial}{\partial \rho} +  \partial_{1} \bar f  \frac{\partial}{\partial s}  \\
 \frac{\partial}{\partial w^2} =& (1+ O(R^{-2})) \frac{\partial}{\partial v^2}+ O(R^{-2})\frac{\partial}{\partial v^1}+ O(1)  \frac{\partial}{\partial \rho} +  \partial_{2} \bar f  \frac{\partial}{\partial s}.  \\
\end{split}
\]
Hence, 
\[ \langle L,  \frac{\partial}{\partial w^a}  \rangle = r^2 \tilde \sigma_{2 a} + O(1).  \]

This proves 
\[  |(L')^\top- L^\top |  =O(R^{-1}).  \]

On the other hand, 
\[ \breve e_4 = \frac{\frac{\partial}{\partial s} - (\partial_a \bar f) \sigma^{ab} \frac{\partial}{\partial w^b}  }{\sqrt{1-|\hat \nabla \bar f|^2}}  \]
where $\hat \nabla$ is the covariant derivative of  $\widehat \Sigma_{R, \bar f}$. 
As a result,
\[  \langle L , \breve e_4  \rangle =\frac{-\partial_2 \bar f}{\sqrt{1- R^{-2}| \tilde \nabla \bar f|^2}}.\] 
On the other hand,
\[  e_4= \frac{ \frac{\partial}{\partial t} + V^a  \frac{\partial}{\partial u^a} } {\sqrt{-g_{tt} - |V|^2 - 2V^ag_{at}}} \]
where
\[ V_{a}= - g_{at} - f_a g_{tt} . \]
Hence,
\[\begin{split} 
 \langle L' ,  e_4  \rangle  = & \frac{- \partial_2 f}{\sqrt{-g_{tt} - |V|^2 - 2V^ag_{at}}}  \\
& =\frac{- \partial_2 \bar f}{\sqrt{1 - \frac{|V|^2}{-g_{tt}}}} + o(R^{-1}).
\end{split}
  \]
Since
\[  V_{a}= - g_{at} - f_a g_{tt} = -g_{at} + \bar f_a  \sqrt{-g_{tt}},\]
it is easy to see that 
\[ \frac{|V|^2}{-g_{tt}} = R^{-2}| \tilde \nabla \bar f|^2 + o(R^{-2}).  \]
Thus
\[  \langle L' ,  e_4  \rangle =\frac{-\partial_2 \bar f}{\sqrt{1- R^{-2}| \tilde \nabla \bar f|^2}} + o(R^{-1}).\] 
\[  \langle L , \breve e_4  \rangle -  \langle L' ,  e_4  \rangle =  o(R^{-1}).\]
\end{proof}
By Lemma \ref{null_compare}, we have
\[ 
\begin{split}
   & E_2(  \Sigma_{R,f},X_R,T_0,L) -  \int_{\Sigma_{R,f}}  \pi_1(L',  n) \,   \, d \Sigma_{R,f}\\
 =& \int_{\Sigma_{R,f}} (\langle L, \breve e_4 \rangle \langle H_{f}, \bar e_4 \rangle - \langle \nabla_{L^T} \bar e_3, \bar e_4 \rangle)   \, d \Sigma_{R,f} . 
    - \int_{\Sigma_{R,f}} (\langle L', e_4 \rangle \langle H_{f},  e_4 \rangle - \langle \nabla_{(L')^T}  e_3, e_4 \rangle )  \, d \Sigma_{R,f}\\
= & o(1).
\end{split}
\]
This finishes the proof of the proposition.
\end{proof}
Let $(L^\top)'$ be the tangential part of $L$ to the image of $X'_R$. Then 
\[  (L^\top)'_a  - L^\top_a =o(1)  \]
since the time function is changed by $o(1)$. Hence
\[  \sigma^{ab}[(L^\top)'_a j_b' - (L^\top)_a  j_b ] =  \sigma^{ab}[(L^\top)'_a  (j_b' - j_b)- ((L^\top)'_a - (L^\top)_a)  j_b ] =o(R^{-2}).  \]
It follows 
\[  \lim_{R \to \infty}E(\Sigma_{R,f},X'_R,T_0,L)=\lim_{R \to \infty}E(\Sigma_{R,f},X_R,T_0,L).  \]
This finishes the proof of the theorem.
\end{proof}

\section{Evolution of total angular momentum and center of mass under the Einstein equation.}

In this section, we study the evolution of the total center of mass and angular momentum of a vacuum asymptotically flat initial data set of order $1$ (see Definition \ref{order_one}) under the vacuum Einstein equation.

Assume that we have an asymptotically flat initial data $(M,g_{ij}(0),k_{ij}(0))$ satisfying the vacuum constraint equation
equation \eqref{constraint}.

We shall fix an asymptotically flat coordinate system on $M$ with respect to $(g_{ij}(0), k_{ij}(0))$ and consider a family $(g_{ij}(t), k_{ij}(t))$ that evolves according to the vacuum Einstein  evolution equation
\begin{equation} \label{evolution}
\begin{split}
\partial _t g_{ij} & =-2N k_{ij}+(\mathcal L _{\gamma} g)_{ij}\\
\partial _t k_{ij} & = -\nabla_i\nabla_jN+N\left(R_{ij} + (tr k) k_{ij} - 2 k_{il} k^l \,_j\right) + (\mathcal L _{\gamma} k)_{ij}
\end{split}
\end{equation}
where $N$ is the lapse function, $\gamma$ is the shift vector and $\mathcal L$ is the Lie derivative.

The metric on the spacetime generated by the vacuum Einstein equation is then
\[-N dt^2+2\gamma_i dt dx^i+g_{ij}dx^i dx^j.\]

Suppoe $N=1+O(r^{-1})$ and $\gamma = O(r^{-1})$, it is easy to see that $\partial _t g_{ij} =O(r^{-2})$ and $\partial _t k_{ij} =O(r^{-3})$.
 As a result, 
\begin{equation}\label{physical_expansion}  
\begin{split}
\sigma_{ab} & = r^2 \tilde \sigma_{ab}+ r \sigma_{ab}^{(1)} + \sigma_{ab}^{(0)}(t) + o(1) \\
|H| & = \frac{2}{r}+ \frac{h^{(-2)}}{r^2}+\frac{h^{(-3)}(t)}{r^3} + o(r^{-3}) \\
\alpha_H & = \frac{\alpha_H^{(-1)} }{r}+ \frac{\alpha_H^{(-2)}(t) }{r^2} + o(r^{-2}) 
\end{split}
\end{equation}
The optimal isometric embedding equation of such data was solved in  \cite{Chen-Wang-Yau1} and the solution satisfies the following asymptotic expansion:
\begin{equation}\label{embedding_expansion}
\begin{split}
X^0 & =(X^0)^{(0)} + \frac{(X^0)^{(-1)}(t)}{r} + o(r^{-1}) \\
X^i  & = r \tilde X^i + (X^i)^{(0)} + \frac{(X^i)^{(-1)}(t)}{r} + o(r^{-1}) \\
T_0 & = (a^0,a^i) + \frac{T_0^{(-1)}(t)}{r} + O(r^{-2}).
\end{split}
\end{equation}
where $ (a^0,a^i)$ points to the direction of the total energy-momentum 4-vector which is a constant in this case.

As a result, the corresponding data on the image of isometric embedding satisfies
\begin{equation}\label{reference_expansion}
\begin{split}
|H_{0}| & = \frac{2}{r}+ \frac{h_{0}^{(-2)}}{r^2}+\frac{h_{0}^{(-3)}(t)}{r^3} + o(r^{-3}) \\
\alpha_{H_0} & = \frac{\alpha_{H_{0}}^{(-1)} }{r}+ \frac{\alpha_{H_{0}}^{(-2)}(t) }{r^2} + o(r^{-2}).
\end{split}
\end{equation}
Since vacuum constraint equation is preserved by vacuum Einstein evolution equation,  by Theorem \ref{finite}, the total center of mass and total angular momentum is well defined for each $t$.
\begin{remark}
From the above discussion and equation \eqref{ADM_e} and \eqref{ADM_p}, the ADM energy momentum is conserved by the Einstein equation in this case. 
\end{remark}

We first prove the following theorem.
\begin{theorem} \label{thm_basic_variation}
Under the vacuum Einstein evolution equation \eqref{evolution} with lapse function $N=1+O(r^{-1})$ and shift vector $\gamma=\gamma^{(-1)} r^{-1}+ O(r^{-2})$, we have
\[ \partial_t   \left [  \frac{1}{8\pi} \int_{\Sigma_r}
 X^i\rho \,\,d \Sigma_r  \right  ] = \frac{p^i}{a^0} + O(r^{-1}) \] for $i=1, 2, 3$. 
  
\end{theorem}
\begin{proof}
By the expansion \eqref{physical_expansion}, \eqref{embedding_expansion}, \eqref{reference_expansion}, 
\[ \partial_t   \left [  \frac{1}{8\pi} \int_{\Sigma_r}
 X^i\rho \,\,d \Sigma_r  \right  ] =  \frac{1}{8\pi} \int_{\Sigma_r}
 X^i \partial_t \rho \,\,d \Sigma_r + O(r^{-1}). \]

By the definition of $\rho$ and the expansion \eqref{physical_expansion}, \eqref{embedding_expansion}, \eqref{reference_expansion}, 
\[  \partial_t \rho =\frac{\partial_t (|H_0|-|H|) }{a^0}- \frac{(|H_0|-|H|)\sum_{j} a^j  \partial_t T_0^{(-1)}(t)^j} {(a^0)^3r}   + O(r^{-4}). \]

As a result,
\[ \begin{split}  \partial_t   \left [  \frac{1}{8\pi} \int_{\Sigma_r}
 X^i\rho \,\,d \Sigma_r  \right  ] =  & \frac{1}{8 \pi } \int_{\Sigma_r} X^i \frac{\partial_t (|H_0|-|H|) }{a^0}  \,\,d \Sigma_r \\
  &  - \frac{\sum_{j} a^j  \partial_t T_0^{(-1)}(t)^j  }{8\pi  (a^0)^3} \int_{\Sigma_r}\frac{X^i (|H_0|-|H|)}{r}  \,\,d \Sigma_r + O(r^{-1}).
\end{split}\]

By Theorem \ref{finite},
\[  \int_{\Sigma_r}\frac{X^i (|H_0|-|H|)}{r}  \,\,d \Sigma_r  = O(r^{-1}).   \]
It follows that 
\[ \partial_t   \left [  \frac{1}{8\pi} \int_{\Sigma_r}
 X^i\rho \,\,d \Sigma_r  \right  ] =\frac{1}{8 \pi} \int_{S^2}  \frac{\tilde X^i \partial_t (h_{0}^{(-3)}(t) -h^{(-3)}(t))}{a^0}  dS^2 +O(r^{-1}).\]

Let us study the two terms separately. First, we prove the following lemma
\begin{lemma} \label{lemma_evo_1}
Under the Einstein vacuum equation with lapse function $N=1+O(r^{-1})$ and shift vector $\gamma=O(r^{-1})$,
\[ \partial_t \left [ \int_{S^2} \tilde X^i  h_{0}^{(-3)}(t)  dS^2 \right ] =0.  \]
\end{lemma}
\begin{proof}
First, we show that 
\[  \partial_t (|H_{0}|^2 -\hat h_0^2) = O(r^{-5}). \]
where $\hat h_0$ is  the mean curvature of the image of the isometric of $\sigma_{ab}$ into $\R^3$.

Let $X'$ be the isometric embedding of $\sigma_{ab}$ into $\R^3$.  $X'$ satisfies the isometric embedding equation 
\begin{equation} \label{iso_1}
 \langle dX' ,dX' \rangle_{\R^3}= \sigma. 
\end{equation}

On the other hand, $X^i$ satisfies the isometric embedding equation 
\[  \langle  dX, dX \rangle_{\R^3}  = \sigma- d X^0\otimes d X^0; \]
It follows that 
\[ 
(X')^i-X^i  =O(r^{-1})
\]
since $d X^0\otimes d X^0 = O(1)$.

Taking the derivative of equation \eqref{iso_1} with respect to $t$ and considering its leading order term, we derive that $(Y')^k = \partial_t ((X')^k)^{(-1)}$, $k=1,2,3$ are functions on $S^2$ that solve the linearized isometric embedding equation 
\[   \tilde  \nabla_a \tilde X^k \tilde \nabla _b (Y')^k + \tilde  \nabla_b \tilde X^k \tilde \nabla _a (Y')^k   = \partial_t \sigma_{ab}^{(0)}  \]
Similarly, $Y^k=\partial_t (X^k)^{(-1)}$, $k=1,2,3$  are functions on $S^2$ that solve the linearized isometric embedding equation 
\[   \tilde  \nabla_a \tilde X^k \tilde \nabla _b (Y)^k + \tilde  \nabla_b \tilde X^k \tilde \nabla _a (Y)^k   = \partial_t \sigma_{ab}^{(0)} -
\tilde \nabla_a X^{(0)}   \tilde \nabla_b( \partial_t X^{(0)}) - \tilde \nabla_b X^{(0)}   \tilde \nabla_a( \partial_t X^{(0)})  .\]
However, since $\partial_t X^{(0)} =0$,  we have $Y= Y'$ and 
\[ 
\partial_t [(X')^i-X^i ] =O(r^{-2}). 
\]
For any function $f$ on $\Sigma_r$,
\[
\begin{split}
(\partial_t  \Delta) f = & \partial_t [ \sigma^{ab}(\partial_a \partial _b -  \Gamma_{ab}^c \partial_c)] f  \\
=& (\partial_t \sigma^{ab}) (\partial_a \partial _b -  \Gamma_{ab}^c \partial_c) f  - \sigma^{ab}( \partial_t \Gamma_{ab}^c )\partial_c f \\
= & O(r^{-4} f)
\end{split}
\] As a result,
\[\begin{split}
    &\partial_t (|H_{0}|^2-\hat h_0^2)\\
=&  \partial_t [\Delta (X^i - (X')^i)\Delta (X^i+ (X')^i) ]-       2 \Delta X^0  \partial_t (\Delta X^0) \\
= & O(r^{-5}).
\end{split}\]

The mean curvature of the image surface satisfies \[  \hat h_0^2 = \sum_k(\Delta (X')^k)^2. \] Hence,
\begin{equation}\label{mean_curvature_derivative} 
\begin{split}
\hat h_0 \partial_t \hat h_0 =  & \sum_k \Delta (X')^k  (\partial_t \Delta) (X')^k + \Delta (X')^k \Delta (\partial_t (X')^k). \\
\end{split}
\end{equation}

The left hand side of \eqref{mean_curvature_derivative} is 
\[ \frac{2  \partial_t  \hat  h_0^{(-3)} }{r^4} +O(r^{-5}) . \]

Denote $T_{ab}= \partial_t \sigma_{ab}^{(0)}$.
We compute
\[
\begin{split}
(\partial_t  \Delta) (X')^k = & \partial_t [ \sigma^{ab}(\partial_a \partial _b -  \Gamma_{ab}^c \partial_c)] (X')^k  \\
=& (\partial_t \sigma^{ab}) (\partial_a \partial _b -  \Gamma_{ab}^c \partial_c) (X')^k  - \sigma^{ab}( \partial_t \Gamma_{ab}^c )\partial_c (X')^k \\
= & \frac{  \tilde \sigma^{ab}T_{ab}X^k- \tilde \sigma^{ab}( {\partial_t \Gamma_{ab}^c}^{(-2)} )\partial_c \tilde X^k}{r^3}  +o(r^{-3})
\end{split}
\]  and $ \Delta (X')^k = \frac{-2 \tilde X^k}{r} + O(r^{-2})$.
Therefore, the first term on the right hand side of \eqref{mean_curvature_derivative} is,
\[   \sum_ k \Delta (X')^k  (\partial_t \Delta) (X')^k  = \frac{-2\tilde \sigma^{ab}T_{ab}}{r^4} + O(r^{-5}).\]
The second term on the right hand side of \eqref{mean_curvature_derivative} can be expanded as
\[ \sum_k \Delta (X')^k \Delta (\partial_t (X')^k) = \frac{\sum_k -2 \tilde (X')^k \tilde \Delta (Y')^k}{r^4} + O(r^{-5})\ \] 
where $(Y')^k, k=1, 2, 3$ are functions on $S^2$ that solve  the linearized isometric embedding equation
\[ \sum_k \tilde \nabla_a \tilde X^k \tilde \nabla_b (Y')^k + \tilde \nabla_b \tilde X^k \tilde \nabla_a (Y')^k= T_{ab}.\]
Thus, from \eqref{mean_curvature_derivative}, we deduce that 
\begin{equation}
\label{dth_0} \partial_t \hat h^{(-3)}_0=-\tilde \sigma^{ab}T_{ab}-\tilde X^k \tilde \Delta  (Y')^k
\end{equation}
Following the proof of  Proposition \ref{reference_finite} with $\sigma^{(1)}_{ab}$ replaced by $T_{ab}$, we have
\[   \int_{S^2}  \tilde X^i \tilde X^k \tilde \Delta  (Y')^k  dS^2= \int_{S^2}  \tilde X^i \sigma^{ab}T_{ab} dS^2.  \]
This finishes the proof of the lemma.
\end{proof}

For the physical term $\int \tilde X^i \partial_t  h ^{(-3)} $, we prove the following lemma,   

\begin{lemma} \label{lemma_evo_2}
Under the vacuum Einstein evolution equation \eqref{evolution} with lapse function $N=1+O(r^{-1})$ and shift vector $\gamma=\frac{\gamma^{(-1)}}{r}+O(r^{-2})$,
\[  \frac{1}{8 \pi}\int_{S^2} \tilde X^i \partial_t  h^{(-3)} d S^2 = -p^i. \]
\end{lemma}
\begin{proof}

Since 
\[  \partial_t k_{ij}= -\nabla_i \nabla_j N+N(R_{ij} + tr k k_{ij} - 2 k_{il} k^l \,_j)+ (\mathcal L_{\gamma}k)_{ij}= O(r^{-3})\] and $tr_{\Sigma_r} k$ is of lower order, in studying $\partial_t |H|$,
it suffices to look at the time derivative of mean curvature $\hat h$ of $\Sigma_r$ with respect to the 3-metric $g_{ij}(t)$.

In spherical coordinate $\{r, u^a \}$, the metric is of the form
\[  g= g_{rr} dr^2 + 2 g_{ra}dr du^a  + g_{ab} du^a du^b\]
where $g_{rr}=1+O(r^{-1})$, $g_{ra} = O(1)$ and $g_{ab}= \sigma_{ab}$ is the induced metric on $\Sigma_r$. Denote $\beta _a = g_{ra}.$ The unit normal vector of the surface $\Sigma_r$ is 
\[  n = \frac{\frac{\partial}{\partial r}-\beta^a\frac{\partial}{\partial u^a}}{\sqrt{g_{rr} - |\beta|^2}}.  \]
The mean curvature of the surface $\Sigma_r$ with respect to the 3-metric $g_{ij}(t)$ is 
\[ \hat h =   \frac{\partial_r \ln \sqrt{det(\sigma_{ab})} - \nabla ^a \beta_a }{\sqrt{g_{rr} - |\beta|^2}}.   \]
As a result, 
\begin{equation}
\begin{split}
\partial_t \hat h = & \partial_t  \left [ \frac{\partial_r \ln \sqrt{det(\sigma_{ab})} - \nabla ^a \beta_a }{\sqrt{g_{rr} - |\beta|^2}} \right   ] \\
=&  \frac{\partial_t \partial_r \ln \sqrt{det(\sigma_{ab})} -\partial_t \nabla ^a \beta_a}{\sqrt{g_{rr} - |\beta|^2}} + \partial_t \frac{1}{\sqrt{g_{rr} - |\beta|^2}} [\partial_r \ln \sqrt{det(\sigma_{ab})} - \nabla ^a \beta_a ]. 
\end{split}
\end{equation}
 Let
\[ S_{ij} =  (\mathcal L_{\gamma}g)_{ij}  \] 
be the deformation tensor for the shift vector $\gamma$, then
\[  \partial_t g_{ij}= -2Nk_{ij}+S_{ij}\]
and
\begin{equation}
\partial_t \hat h = - \partial_r tr_{\Sigma_r} k +2( \nabla^a k_{ra} +\frac{k_{rr}}{r}) - ( \frac{tr_{\Sigma_r} S}{r} + \nabla^a S_{ra} +\frac{S_{rr}}{r}) +O(r^{-4}).
\end{equation}
We claim that 

\begin{align}
\label{eq_change_k} \int_{\Sigma_r}  X^i [ - \partial_r tr_{\Sigma_r} k +2( \nabla^a k_{ra} +\frac{k_{rr}}{r})] d\Sigma_r & =  -p^i +O(r^{-1})\\
\label{eq_change_s} \int_{\Sigma_r} X^i ( \frac{tr_{\Sigma_r} S}{r} + \nabla^a S_{ra} +\frac{S_{rr}}{r}) d\Sigma_r & =  O(r^{-1}).
\end{align}

To prove equation (\ref{eq_change_k}), we first relate the last term on the left hand side to the other two terms using the vacuum constraint equation. 

Let $\pi(t)$ be the conjugate momentum of $(M, g_{ij}(t), k_{ij}(t))$. That is,
\begin{equation}\label{conjugate} \pi = k -(tr_g k) g. \end{equation}
The vacuum constraint equation states that $\pi$ is divergence free. Namely, 
\[  \nabla_g^i \pi_{ij} =0. \]
In particular, we will use leading term (of $O(r^{-3})$) of  the equation
\[  \nabla_g^i \pi_{ir} =0 \] in spherical coordinates.  In the appendix, it is proved that 
\begin{equation}
\label{eq:constraint2}
\nabla_g^i \pi_{ir} =\partial_r \pi_{rr} + \nabla^a \pi_{ar} + \frac{2\pi_{rr}}{r} -\frac{g^{ab}\pi_{ab}}{r}+ O(r^{-4}).
\end{equation}
By equation \eqref{conjugate} and the vacuum constraint equation, we have 
\[  -\partial_r tr _{\Sigma_r}k -\frac{tr _{\Sigma_r}k}{r}+ \nabla^a k_{ar} + \frac{2k_{rr}}{r}= O(r^{-4}) \]
As a result, 
\[  - \partial_r tr_{\Sigma_r} k +2( \nabla^a k_{ra} +\frac{k_{rr}}{r}) =  \nabla^a k_{ar} + \frac{tr_{\Sigma_r} k}{ r}+ O(r^{-4}) \] 
We derive that
\begin{equation} 
 \int_{\Sigma_r}  X^i [ - \partial_r tr_{\Sigma_r} k +2( \nabla^a k_{ra} +\frac{k_{rr}}{r})] d\Sigma_r  = \int_{\Sigma_r}  X^i ( \frac{tr_{\Sigma_r} k}{r} + \nabla^a k_{ra} ) d\Sigma_r + O(r^{-1}).  
\end{equation}

It is proved in  \cite{Wang-Yau3} that the ADM linear momentum of an asymptotically flat manifold can be computed as limit of quasi-local energy as 
\[ \frac{1}{8 \pi} \lim_{r \to \infty} \int_{\Sigma_r} X^i  div_{\Sigma} (\alpha_H)_a d \Sigma_r. \]
Moreover, from the same paper,
\[ (\alpha_H)_a = -k_{\nu a}+ \nabla_a \frac{tr_{\Sigma_r}k }{|H|}  \text{ and } k_{\nu a} = k_{ra} + O(r^{-2}). \]
Hence, the linear momentum is 
\[  
\begin{split}
     \lim_{r \to \infty} \int_{\Sigma_r} X^i  \nabla^a (\alpha_H)_a d{\Sigma_r}
= & \lim_{r \to \infty} \int_{\Sigma_r} X^i  \nabla^a ( -k_{\nu a}+ \nabla_a \frac{tr_{\Sigma_r}k }{|H|} )d{\Sigma_r}\\
= & \lim_{r \to \infty} \int_{\Sigma_r} X^i  \nabla^a ( -k_{r a}+ \nabla_a \frac{tr_{\Sigma_r }k }{|H|} )d{\Sigma_r}\\
=& \lim_{r \to \infty} \int_{\Sigma_r}\left(- X^i  \nabla^a  k_{r a}+ X^i \Delta \frac{tr_{\Sigma_r }k }{|H|} \right)d{\Sigma_r}\\
=& \lim_{r \to \infty} \int_{\Sigma_r}\left(- X^i  \nabla^a  k_{r a}- X^i \frac{2 tr_{\Sigma_r }k }{|H|r^2} \right) d{\Sigma_r}\\
=& \lim_{r \to \infty} \int_{\Sigma_r}\left(- X^i  (\nabla^a  k_{r a} +  \frac{ tr_{\Sigma_r }k }{r})\right) d{\Sigma_r}\\
\end{split}
\]
This proves equation \eqref{eq_change_k}.  To prove equation (\ref{eq_change_s}), we first observe that 
\[ S_{ij} =  \nabla^0_i \gamma_j+\nabla^0_j \gamma_i + O(r^{-3})\]
where $\nabla^0$ is the covariant derivative with respect to the flat metric.

In spherical coordinates, we write $\gamma$ as 
\[ \gamma = \gamma_r dr + \gamma_a d u^a \]
with $\gamma_r = O(r^{-1})$ and $ \gamma_a=O(1)$.
In spherical coordinates, $S$ is given by 
\[ S_{rr}= 2 \nabla^0_r \gamma_r+ O(r^{-3}) =2\partial_r \gamma_r +O(r^{-3}) \]
\[ S_{ra}= \nabla^0_a\gamma_r +\nabla^0_r \gamma_a+O(r^{-2}) = \partial_a \gamma_r -2\frac{\gamma_a}{r }+O(r^{-2}) \]
\[ S_{ab} =\nabla^0_a\gamma_b +\nabla^0_b \gamma_a + O(r^{-1}) = \nabla^{\Sigma_r}_a \gamma_b + \nabla^{\Sigma_r}_b \gamma_a +2r \tilde \sigma_{ab}\gamma_r+ O(r^{-1}). \]
Hence,
\[
\begin{split}
    &  \frac{tr_{\Sigma} S}{r} + \nabla^a S_{ra} +\frac{S_{rr}}{r}  \\
= &  \frac{2 div_{\Sigma_r} \gamma_b +4\frac{\gamma_r}{r}}{r}+ \Delta   \gamma_r  -\frac{2 div_{\Sigma_r} \gamma_b }{r}  +\frac{2\partial_r \gamma_r}{r} +O(r^{-4})
\end{split}
 \]
As a result, if  $\gamma_r=\gamma_r^{(-1)}r^{-1}+O(r^{-2})$ then we have 
\[  \frac{tr_{\Sigma_r} S}{r} + \nabla^a S_{ra} +\frac{S_{rr}}{r} = \frac{(\tilde \Delta+2) \gamma_r^{(-1)}  }{r^{-3}}+ O(r^{-4}) \]
Integration by parts, we derive
\[   \int_{\Sigma_r} 2 X^i ( \frac{tr_{\Sigma_r} S}{r} + \nabla^a S_{ra} +\frac{S_{rr}}{r}) d\Sigma_r  = \int_{S^2} \tilde X^i (\tilde \Delta+2) \gamma_r^{(-1)} dS^2+ O(r^{-1}) = O(r^{-1}).  \]
\end{proof}

 Theorem \ref{thm_basic_variation} follows from Lemma \ref{lemma_evo_1} and Lemma \ref{lemma_evo_2}.
\end{proof}

\begin{lemma} \label{lemma_invariance}
\[  \partial_t[\int_{\Sigma_r}
[X^i \nabla_a X^j - X^j \nabla_a X^i ]\sigma^{ab} j_b\, d\Sigma_r]=O(r^{-1}).  \]
\end{lemma}
\begin{proof}
It is easy to see that $j_b$  is of the order $O(r^{-1})$ and $\partial_t j_b$ is of order $O(r^{-2})$. Moreover, 
\begin{equation}
\begin{split}
X^i \nabla_a X^j - X^j \nabla_a X^i =& r^2(\tilde X^i  \tilde{\nabla}_a \tilde{X}^j  -\tilde X^j \tilde{\nabla}_a  \tilde{X}^i )+O(r)\\
\partial_t(X^i \nabla_a X^j - X^j \nabla_a X^i )=&O(1).
\end{split}
\end{equation}
Using the above expansion, we derive
\[  \partial_t[\int_{\Sigma_r}
[X^i \nabla_a X^j - X^j \nabla_a X^i ]\sigma^{ab} j_b\, d\Sigma_r]=\int_{S^2} [\tilde X^i \tilde \nabla_a \tilde X^j - \tilde X^j \tilde \nabla_a \tilde X^i ]\tilde \sigma^{ab} \partial_t j_b^{(-2)} dS^2+ O(r^{-1}).  \]

Let's look at the four terms in $ j_b= -\rho{\nabla_b \tau }+ \nabla_b[ \sinh^{-1} (\frac{\rho\Delta \tau }{|H||H_0|})]-(\alpha_H)_b + (\alpha_{H_0})_b $ individually. Among the four terms, $ \nabla_b[ \sinh^{-1} (\frac{f\Delta \tau }{|H||H_0|})]$ is a closed one-form . As a result, its integration against $\tilde X^i  \tilde{\nabla}_a \tilde{X}^j  -\tilde X^j \tilde{\nabla}_a  \tilde{X}^i $ vanish. Moreover, while $(\alpha_{H_0}) $ is not closed, by the expansion of  $(\alpha_{H_0})_b $ in  Lemma 5 of \cite{Chen-Wang-Yau1} and the expansion of data in equation \eqref{embedding_expansion} and \eqref{reference_expansion},
\[  \partial_t (\alpha_{H_0})_b  = \frac{1}{2r^2} \tilde \nabla_b [ (\tilde \Delta +2)\partial_t (X^0)^{(-1)}]+O(r^{-3}). \]
Thus, $ (\alpha_{H_0})_b$ has no contribution either.

For the first term, we derive
\[  \partial_t (\rho \nabla_b  \tau) = -r^{-2} \sum_k [ a^k (\partial_t \rho^{(-3)}) \tilde \nabla_b  \tilde X^k +   \rho^{(-2)} \tilde \nabla_b \partial_t  (T_0^{(-2)})^k \tilde X^k  ]+O(r^{-3})\]
and 
\[  
\begin{split}
 & \int_{S^2} [\tilde X^i \tilde \nabla_a \tilde X^j - \tilde X^j \tilde \nabla_a \tilde X^i ]\tilde \sigma^{ab}\sum_k [ a^k (\partial_t \rho^{(-3)}) \tilde \nabla_b  \tilde X^k -   \rho^{(-2)} \tilde \nabla_b \partial_t  (T_0^{(-2)})^k \tilde X^k  ] dS^2\\
= & \int_{S^2}(\tilde X^j  a^i -\tilde X^i  a^j) (\partial_t \rho^{(-3)}) dS^2 + \int_{S^2}[\tilde X^j  \partial_t  (T_0^{(-2)})^i-\tilde X^i  \partial_t  (T_0^{(-2)})^j]\rho^{(-2)} dS^2
\end{split}
\]
By Theorem \ref{finite}, we have
\[ \int_{S^2} \tilde X^i \rho^{(-2)} dS^2= 0  \]
By Theorem 9.2, we have
\[[\frac{1}{8\pi} \int_{S^2} \tilde X^i  \partial_t\rho^{(-3)}  \,\,d S^2]=\frac{p^i}{a^0}. \]
It follows that  $ (\rho \nabla_b  \tau) $ has no contribution either.

Finally, 
\[\begin{split}  
   & \partial_t \int_{\Sigma_r} ( X^i \nabla_a   X^j  - X^j \nabla_a   X^i)   \sigma^{ab}   (\alpha_H)_b d \Sigma_r\\
= & - \partial_t \int_{\Sigma_r} ( X^i \nabla_a   X^j  - X^j \nabla_a   X^i)   \sigma^{ab} k_{\nu b} d \Sigma_r + O(r^{-1})
\end{split}
\]
since $(\alpha_H)_b$ and $-k_{\nu b}$ differ by a gradient vector filed.

Consider the timelike cylinder $\mathcal C$ defined by $r=c$ in the solution of the Einstein equation.  Let $\gamma$ be the induced metric on $\mathcal C$ and  $\pi$ be the conjugate momentum of $\mathcal C$. The cylinder $\mathcal C$ is foliated by the level set of $t$. On each level set of $t$, we have the tangential vector field $X^i \nabla_a X^j  - X^j   \nabla_a X^i$. This gives a vector field $K'$ on $\mathcal C$. 

We will use the divergence theorem to prove the conservation we need. Let us consider the vector field $Z=\pi(K',\cdot)$ on $\mathcal C$. Let $\Omega_c$ be the portion of $\mathcal C$ bounded between $t=0$ and $t=c$. Denote the surface $\mathcal{C}\cap \{t=t_0\}$ by $\Sigma_{t_0}$. By the divergence theorem, we have
\[ \int_{\Sigma_c} Z \cdot n d \Sigma_c=  \int_{ \Sigma_0} Z \cdot n  d\Sigma_0+ \int_{\Omega_c} div_{\gamma} Z. \]
However,
\[ 
\begin{split}
\int_{\Sigma_c} Z \cdot n=& \int_{t=c} \pi(K', \frac{\partial}{\partial t})  d \Sigma_c\\
= & \int_{\Sigma_c} \langle \nabla_{K'} \frac{\partial}{\partial t} , \nu  \rangle  d \Sigma_c \\
= & -\int_{\Sigma_c} (K')_a \sigma^{ab}k_{\nu b} d \Sigma_c.
\end{split}
 \]
Similarly,
\[ \int_{\Sigma_0} Z \cdot n = -\int_{t=0} (K')_a \sigma^{ab}k_{\nu b}d\Sigma_0. \]
As a result, to show that $\lim_{R \to \infty} \int_{t=c, r=R} (K')_a \sigma^{ab}k_{\nu b} $ is independent of $c$, it suffices to show that 
\[  \lim_{R \to \infty }\int_{\Omega_c} div_{\gamma} Z  =0 .\]
We have
\[ 
\begin{split}
 div_{\gamma} Z = & \pi_{ij} \langle \nabla_{\partial_i} K', \partial_j \rangle  \\
=& \pi_{ij} \langle \nabla_{K'}  \partial_i , \partial_j \rangle + O(R^{-3})\\
=& \frac{1}{2R} \sigma^{ab} K'(\sigma_{ab})+ O(R^{-3})\\
=& \frac{1}{2R} K'( \ln \sqrt{det \sigma}) + O(R^{-3})\\
=& \frac{1}{2R} K'( \ln \sqrt{\frac{det \sigma}{R^2\det \tilde \sigma}} ) + O(R^{-3})\\
\end{split}
\]
As a result, 
\[
    \begin{split} 
     \int_{\Sigma_{c'}} div_{\gamma} Z    =  & \int_{\Sigma_{c'}} \frac{1}{2R} K'( \ln \sqrt{\frac{det \sigma}{R^2\det \tilde \sigma}})   d \Sigma_{c'}+O(R^{-1}) \\
=  & \frac{1}{2R} \int_{\Sigma_{c'}} (-div_{\sigma}K') \ln \sqrt{\frac{det \sigma}{R^2\det \tilde \sigma}}d\Sigma_{c'} +O(R^{-1})
\end{split}
\]
Moreover, $div_{\sigma}K'=O(R^{-1})$ and $ \ln \sqrt{\frac{det \sigma}{R^2\det \tilde \sigma}} =1 +O(R^{-1}) $.
It follows that 
\[  \lim_{R \to \infty }\int_{\Omega_c} div_{\gamma} Z  =0. \]
\end{proof}

We are now ready to prove the evolution  of total angular momentum and center of mass under the vacuum Einstein evolution equation.

\begin{theorem}\label{thm_variation_center}
Suppose $(M, g, k)$ is anasymptotically flat  initial data of order $1$ satisfying the vacuum constraint equation. Let $(M, g(t), k(t) )$ be the solution to the initial value problem $g(0)=g$ and $k(0)=k$ for the vacuum Einstein evolution equation with lapse function $N=1+O(r^{-1})$ and shift vector $\gamma=\gamma^{(-1)}r^{-1}+O(r^{-2})$. The total center of mass $C^i(t)$ and total angular momentum $J_i(t)$ of $(M, g(t), k(t))$ satisfy
\[\begin{split}
\partial_t C^i (t) &=\frac{p^i}{e},\\
\partial_t  J_i (t) &=0, \text{ for }
\end{split}
\]
for $i=1,2,3$ where $(e,p^i)$ is the ADM energy momentum of $(M, g, k )$.
\end{theorem}
\begin{proof}
Recall that the center of mass is 
\[C^i=\frac{1}{8\pi m} \lim_{r \to \infty }E(\Sigma_r,X_r,T_0(r),A(r)(X^i\frac{\partial}{\partial X^0}+X^0\frac{\partial}{\partial X^i}) )\]
where 
\begin{equation}
\begin{split}
X^0 & =(X^0)^{(0)} + \frac{(X^0)^{(-1)}(t)}{r} + o(r^{-1}) \\
X^i  & = r \tilde X^i + (X^i)^{(0)} + \frac{(X^i)^{(-1)}(t)}{r} + o(r^{-1}) \\
T_0 & = (a^0,a^i) + \frac{T_0^{(-1)}(t)}{r} + O(r^{-2})
\end{split}
\end{equation}
and $A(r)$ is a family of Lorentz transformation such that $ T_0({r})=A(r)(\frac{\partial}{\partial X^0})$. Let
\[  A(r) = A^{(0)} + \frac{A^{(-1)}(t)}{r}+o(r^{-1}). \]
We have
\[E(\Sigma_r,X_r,T_0(r),A(r)(X^i\frac{\partial}{\partial X^0}+X^0\frac{\partial}{\partial X^i}) ) =\frac{1}{8\pi }\int_{\Sigma_r}X^i\rho -[A(r)(X^i\frac{\partial}{\partial X^0}+X^0\frac{\partial}{\partial X^i})]^{\top}_a j^a  d\Sigma_r. \] 
By Theorem \ref{thm_basic_variation}, we have
\[ \frac{1}{8\pi m} \partial_t  [\lim_{r \to \infty }\int_{\Sigma_r}X^i\rho  d\Sigma_r ]= \frac{p^i}{e}.  \]
Let us compute $[A(r)(X^i\frac{\partial}{\partial X^0}+X^0\frac{\partial}{\partial X^i})]^{\top}_a$.
\[ 
\begin{split}
   &  [A(r)(X^i\frac{\partial}{\partial X^0}+X^0\frac{\partial}{\partial X^i})]^{\top}_a \\
 = & \langle A(r)(X^i\frac{\partial}{\partial X^0}+X^0\frac{\partial}{\partial X^i}) , \nabla_a X^\alpha \frac{\partial}{\partial X^{\alpha}}  \rangle \\
= & X^i A(r)_{0 \alpha}\nabla_a X^{\alpha} + X^0 A(r)_{i \alpha} \nabla _a X^{\alpha}\\
=& X^i  A^{(0)}_{0 j}\nabla_a X^{j} + X^0 A^{(0)}_{ij} \nabla _a X^{j} + X^i  A^{(0)}_{00}\nabla_a X^{0} + \frac{1}{r}(X^i  A^{(-1)}(t)_{0 j}\nabla_a X^{j} )+o(r).
\end{split}
\]
The first term, $ X^i  A^{(0)}_{0 j}\nabla_a X^{j}$, is of order $r^2$. The other terms are of order $r$. 
We have 
\[  X^i  A^{(0)}_{0 j}\nabla_a X^{j}= \frac{A^{(0)}_{0 j}}{2}[\nabla_a  (X^i X^j) + X^i\nabla_a X^{j} - X^j\nabla_a X^{i} ]. \]
Since $\nabla^a j_a=0$, we have
\[  \partial_t \int_{\Sigma_r} X^i  A^{(0)}_{0 j}\nabla_a X^{j} j^a d \Sigma_r  =   \frac{A^{(0)}_{0 j}}{2}   \partial_t \int_{\Sigma_r} [( X^i\nabla_a X^{j} - X^j\nabla_a X^{i}) j^a ] d \Sigma_r.  \]
By Lemma \ref{lemma_invariance},
\[ \partial_t \int_{\Sigma_r} [( X^i\nabla_a X^{j} - X^j\nabla_a X^{i}) j^a ] d \Sigma_r = O(r^{-1}).  \]

For the second and third term, we have 
\[  \partial_t \int_{\Sigma_r}(X^0 A^{(0)}_{ij} \nabla _a X^{j} + X^i  A^{(0)}_{00}\nabla_a X^{0} )j^a = O(r^{-1}),  \]
since the leading term of $X^0 A^{(0)}_{ij} \nabla _a X^{j} + X^i  A^{(0)}_{00}\nabla_a X^{0} $ and $j^a$ are both independent of $t$. Finally,
\[  \int_{\Sigma_r} \frac{1}{r}(X^i  A^{(-1)}(t)_{0 j}\nabla_a X^{j} ) j^a d\Sigma_r =  A^{(-1)}(t)_{0 j} \int_{S^2} (\tilde X^i  \tilde \nabla_a \tilde X^j )\tilde \sigma^{ab} j^{(-1)}_b + O(r^{-1}).  \]
However, 
\[  \int_{S^2} (\tilde X^i  \tilde \nabla_a \tilde X^j )\tilde \sigma^{ab} j^{(-1)}_b =0. \]
due to $\tilde \nabla^a  j^{(-1)}_a=0$ and  Theorem \ref{finite}.

It follows that 
\[  \partial_t C^i =\frac{p^i}{e}.\]

For the total angular momentum, 
\[ E(\Sigma_r,X_r,T_0(r),A(r)(X^i\frac{\partial}{\partial X^j}-X^j\frac{\partial}{\partial X^i}) ) =  \frac{-1}{8\pi }\int_{\Sigma_r}[A(r)(X^i\frac{\partial}{\partial X^j}-X^j\frac{\partial}{\partial X^i})]^{\top}_a j^a  d\Sigma_r. \]
We compute
\[ 
\begin{split}
    &  A(r)(X^i\frac{\partial}{\partial X^j}-X^j\frac{\partial}{\partial X^i})]^{\top}_a \\
=& (X^i  A^{(0)}_{jk}\nabla_a X^{k} - X^j A^{(0)}_{ik}\nabla_a X^{k}) + X^i  A^{(0)}_{j0}\nabla_a X^{0} - X^j A^{(0)}_{i0}\nabla_a X^{0}  \\
    &+ \frac{1}{r}(X^i  A^{(-1)}_{jk}(t)\nabla_a X^{k} - X^j A^{(-1)}_{ik}(t)\nabla_a X^{k})+o(r),
\end{split}
\]
where $ (X^i  A^{(0)}_{jk}\nabla_a X^{k} - X^j A^{(0)}_{ik}\nabla_a X^{k}) $ is of order $r^2$ and the other terms are $O(r)$. We have
\[  \begin{split} & (X^i  A^{(0)}_{jk}\nabla_a X^{k} - X^j A^{(0)}_{ik}\nabla_a X^{k})  \\
= &\frac{A^{(0)}_{jk}}{2} [\nabla_a (X^iX^k)  + X^i \nabla_a X^{k}-X^k \nabla_a X^{i} ]- \frac{A^{(0)}_{ik}}{2} [ \nabla_a (X^jX^k)  + X^j \nabla_a X^{k}-X^k \nabla_a X^{j} ].
\end{split} \]
Then we use $\nabla^a j_a=0$ and apply Lemma \ref{lemma_invariance} and conclude that 
\[ \partial_t \int_{\Sigma_r}  (X^i  A^{(0)}_{jk}\nabla_a X^{k} - X^j A^{(0)}_{ik}\nabla_a X^{k})  j^a d\Sigma_r  = O(r^{-1}).  \]

The leading term of $X^i  A^{(0)}_{j0}\nabla_a X^{0} - X^j A^{(0)}_{i0}\nabla_a X^{0} $ is independent of $t$ and thus 
\[ \partial_t  \int_{\Sigma_r }(X^i  A^{(0)}_{j0}\nabla_a X^{0} - X^j A^{(0)}_{i0}\nabla_a X^{0} )j^a d\Sigma_r =O(r^{-1}). \]
Finally, for $\frac{1}{r}(X^i  A^{(-1)}_{jk}(t)\nabla_a X^{k} - X^j A^{(-1)}_{ik}(t)\nabla_a X^{k})$, we use $\tilde \nabla^a  j^{(-1)}_a=0$ and Theorem \ref{finite} to conclude that 
\[  \partial_t \int_{\Sigma_r} \frac{1}{r}(X^i  A^{(-1)}_{jk}(t)\nabla_a X^{k} - X^j A^{(-1)}_{ik}(t)\nabla_a X^{k})  j^a d\Sigma_r  = O(r^{-1})  .\]
It follows that 
\[ \partial_t J_i =0. \]
\end{proof}

\section{Dynamical formula with weaker asymptotically flat condition}
In this section, we prove the dynamical formula of total center of mass and angular momentum for vacuum initial data sets  $(M,g,k)$ satisfying 
\begin{equation}
\begin{split}
g=&\delta+O(r^{-1})\\
k=&O(r^{-2}).
\end{split}
\end{equation}
Assume further that for $r$ large, the optimal embedding equation on the coordinate spheres $\Sigma_r$ has an unique solution 
$X_r= (X^0({r}), X^i({r})) $ and observers $T_0(r)$ with the following expansion
\begin{equation} \label{optimal_order}
\begin{split}
X^0({r}) & = O(1) \\
X^i ({r})& = r \tilde X^i + O(1) \\
T_0 ({r})& = (a^0,a^i)+O(r^{-1}).
\end{split}
\end{equation}
where  points to the direction of the total energy-momentum 4-vector.  

We can define the total center of mass $C^i$ and total angular momentum $J_i$ of the initial data $(M,g,k)$ using the optimal embedding $(X_r,T_0(r))$
as in Definition 3.1. However, $C^i$ and $J_i$ may be finite or infinite.  
\begin{remark}
We will address the existence and uniqueness of solution of optimal embedding equation at spatial infinity in a forthcoming paper. 
\end{remark}
\begin{theorem}  \label{cor_variation_center}
Let $(M, g, k)$ be a vacuum initial data set satisfying 
\begin{equation}
\begin{split}
g=&\delta+O(r^{-1})\\
k=&O(r^{-2}).
\end{split}
\end{equation} 

Let $(M, g(t), k(t) )$ be the solution to the initial value problem $g(0)=g$ and $k(0)=k$ for the vacuum Einstein equation with lapse function $N=1+O(r^{-1})$ and shift vector $\gamma=\gamma^{(-1)} r^{-1}+O(r^{-2})$.

Suppose for $r$ large, the optimal embedding equation on each coordinate sphere $\Sigma_{r,t}$ has an unique solution $(X(r,t),T_0(r,t))$ with expansion in equation \eqref{optimal_order} and  that the total center of mass $C^i$ and total angular momentum $J_i$ of the initial data $(M, g, k)$ are finite. Then the total center of mass $C^i(t)$ and total angular momentum $J_i(t)$ of $(M, g(t), k(t))$ satisfy
\[
\begin{split}
\partial_t C^i (t)= &  \frac{p^i}{e},\\
\partial_t J_{i} (t) = &  0
\end{split} 
\]
for $i=1, 2, 3$ where $(e,p^i)$ is the ADM energy momentum of $(M, g, k)$.\end{theorem}
\begin{proof}
We have
$\partial _t g_{ij} =O(r^{-2})$ and $\partial _t k_{ij} =O(r^{-3})$.
 As a result, 
\begin{equation}\label{physical_expansion_2}  
\begin{split}
\sigma_{ab} & = r^2 \tilde \sigma_{ab}+O(r), \qquad   \partial_t \sigma_{ab} =O(1) \\
|H| & = \frac{2}{r}+O(r^{-2}), \qquad   \partial_t |H| =O(r^{-3})\\
(\alpha_H)_a & =O(r^{-1}),\qquad    \partial_t (\alpha_H)_a =O(r^{-2}).
\end{split}
\end{equation}
Furthermore, by computing the $t$-derivative of the optimal embedding equation, we have
\begin{equation}\label{embedding_expansion_2}
\begin{split}
X^0 & =O(1) + O(r^{-1}), \qquad \partial_t X^0 =O(r^{-1}) \\
X^i  & = r \tilde X^i + O(1), \qquad  \partial_t X^i =O(r^{-1}) \\
T_0 & = (a^0,a^i) +O(r^{-1}), \qquad  \partial_t T_0 = O(r^{-1}).
\end{split}
\end{equation}

We first prove that 
\[ \partial_t   \left [  \frac{1}{8\pi} \int_{\Sigma_{r,t}} X^i\rho \,\,d \Sigma_{r,t}  \right  ] = \frac{p^i}{a^0} + o(1). \]
The proof is similar to that of Theorem  \ref{thm_basic_variation}. We derive
\[  \partial_t \rho =\frac{\partial_t (|H_0|-|H|) }{a^0}- \frac{(|H_0|-|H|)\sum_{j} a^j \partial_t T_0^j} {(a^0)^3r}   + O(r^{-4}) \]
which implies
\[ 
\begin{split}
  &\partial_t   \left [  \frac{1}{8\pi} \int_{\Sigma_{r,t}} X^i\rho \,\,d \Sigma_{r,t}  \right  ] \\
= & \frac{1}{8\pi} \int_{\Sigma_{r,t}} \frac{X^i\partial_t (|H_0|-|H|) }{a^0}  \,\,d \Sigma_{r,t}-  \frac{ \sum_{j} a^j \partial_t b_j(r,t)}{8\pi (a^0)^3 r} \int_{\Sigma_{r,t}} X^i (|H_0|-|H|)   \,\,d \Sigma_{r,t}\\
= & \frac{1}{8\pi} \int_{\Sigma_{r,t}} \frac{X^i\partial_t (|H_0|-|H|) }{a^0}  \,\,d \Sigma_r  + o(1).
\end{split} \]
The last equality holds due to Lemma \ref{slow}.
 
We compute the two terms separately.  Following the argument in the proof  of Lemma \ref{lemma_evo_1} with $\partial_t \sigma_{ab}^{(0)}$ and $\partial_t \frac{{\Gamma_{ab}^c}^{(-2)}}{r^2}$ replaced by $\partial_t \sigma$ and $\partial_t \Gamma_{ab}^c$, we derive
\[   \frac{1}{8\pi} \int_{\Sigma_{r,t}} \frac{X^i\partial_t |H_0|}{a^0}  \,\,d \Sigma_{r,t}  = O(r^{-1}).   \]
Similarly, following the argument  in the proof  of Lemma \ref{lemma_evo_2}, we derive
\[   \frac{1}{8\pi} \int_{ \Sigma_{r,t}} \frac{X^i\partial_t |H|}{a^0}  \,\,d \Sigma_{r,t}  = -p^i +O(r^{-1}).   \]

We also need the following 
\begin{equation}  \label{weak_angular_derivative}
 \partial_t[\int_{ \Sigma_{r,t}}
[X^i \nabla_a X^j - X^j \nabla_a X^i ]\sigma^{ab} j_b\, d \Sigma_{r,t}]=o(1).  \end{equation}
Following the proof of Lemma \ref{lemma_invariance}, we derive
\[  \partial_t[\int_{ \Sigma_{r,t}}
[X^i \nabla_a X^j - X^j \nabla_a X^i ]\sigma^{ab} j_b\, d\Sigma_r]=\int_{ \Sigma_{r,t}}
[X^i \nabla_a X^j - X^j \nabla_a X^i ]\sigma^{ab} (\partial_t  j_b)\, d \Sigma_{r,t} + O(r^{-1}).  \]
Recall that 
\[  j_b= -\rho{\nabla_b \tau }+ \nabla_b[ \sinh^{-1} (\frac{\rho\Delta \tau }{|H||H_0|})]-(\alpha_H)_b + (\alpha_{H_0})_b. \]
We can prove equation \eqref{weak_angular_derivative} using the same argument as in the proof of  Lemma \ref{lemma_invariance} with 
 $  \frac{\partial_t (X^0)^{(-1)}}{r}$  and $\frac{\partial_t T_0^{(-1)}}{r}$  replaced by $\partial_t X^0$ and $\partial_t T_0$, respectively and applying Lemma \ref{slow} instead of Theorem \ref{finite}.

Finally, we prove that 
\begin{equation} \label{weak_variation}
\begin{split}
\partial_t  E(\Sigma_{r,t}, X(r,t), T_0(r,t), A(r,t)(X^i\frac{\partial}{\partial X^0}+X^0\frac{\partial}{\partial X^i}))\ &=\frac{p^i}{e}+o(1)\\
\partial_t   E(\Sigma_{r,t}, X(r,t), T_0(r,t), A(r,t)(X^i\frac{\partial}{\partial X^j}-X^j\frac{\partial}{\partial X^i}))\ &=o(1)
\end{split}
\end{equation}
where $T_0(r,t)=A(r,t)\frac{\partial}{\partial X^0}$ for a family of Lorentz transformation $A(r,t)$. 

Equation \eqref{weak_variation} follow from the same argument used in the proof  of Theorem \ref{thm_variation_center}. Again, we need to  replace $\frac{A^{(-1)}}{r}$ by $A - A^{(0)}$ and use Lemma \ref{slow} instead of Theorem \ref{finite}. This finishes the proof of the theorem.
\end{proof}
\section{Conclusion}
The new total angular momentum and total center of mass on asymptotically flat  initial  data sets satisfy the following properties:

1. The definition only depends on the geometric data $(g, k)$ and the foliation of surfaces at infinity, and in particular does not depend on the asymptotically flat coordinate system or the existence of asymptotically Killing field on the initial data set.
 
2. The definition gives a element in the dual space of the Lie algebra of $SO(3,1)$. In particular, the same formula works for total angular momentum and total center of mass, and the only difference is to pair this element with either a rotation Killing field or a boost Killing field. 

3. It is always finite on any asymptotically flat initial data set of order 1. 

4. Both the total angular momentum and total center of mass vanish on any spacelike hypersurface in the Minkowski spacetime.

5. They satisfy conservation law. In particular, the total angular momentum on any strictly spacelike hyperurface of the Kerr spacetime is the same. 

6. Under the vacuum Einstein evolution of initial data sets, the total angular momentum is conserved and the total center of mass obeys the dynamical formula $\partial_t C^i(t)=\frac{p^i}{e}$ where $p^i$ is the ADM linear momentum and $e$ is the ADM energy.

\appendix
\section{The constraint equation}
We prove equation \eqref{eq:constraint1}  and \eqref{eq:constraint2} regarding the constraint equation here. First we have the following lemma on the connection  coefficients of an asymptotically flat metric in spherical coordinates.
\begin{lemma}
For an asymptotically flat metric of the form
\[g_{ij}= \delta_{ij} + O_1(r^{-1}) \]
the connection coefficients in spherical coordinates satisfy
\[
\begin{split}
\Gamma^a_{bc} & = \tilde \Gamma^a_{bc} + O(r^{-1})\\
\Gamma^r_{bc} & = -r \tilde \sigma_{ab} + O(1)\\
\Gamma^b_{ar} & = \frac{\delta^b_a}{r} + O(r^{-2})\\
\Gamma^b_{ar} & = O(r^{-1})\\
\Gamma^a_{rr} & = O(r^{-3})\\
\Gamma^r_{rr} & = O(r^{-2})\\
\end{split}
\]
\end{lemma}
\begin{proof}
Straightforward calculations. 
\end{proof}
\begin{lemma}
Let $(M,g,k)$ be an asymptotically flat initial data set of the form
 \[g_{ij}= \delta_{ij} + O_1(r^{-1}), \qquad  k= O(r^{-2})\]
satisfying the vacuum constraint equation and $\pi$ be the conjugate momentum, then
\begin{equation}
\begin{split}
\nabla^i_g \pi_{ir} =&  \partial_r \pi_{rr} + \nabla^a \pi_{ar} + \frac{2\pi_{rr}}{r} -\frac{g^{ab}\pi_{ab}}{r}+O(r^{-4})\\
\nabla^i_g \pi_{ra} =&  \partial_r \pi_{ar}  - \frac{\pi_{ar}}{r} + \nabla^b_{\sigma} \pi_{ba} + \frac{3\pi _{ar}}{r}+ O(r^{-3})
\end{split}
\end{equation}
\end{lemma}
\begin{proof}
For asymptotically flat initial data set of order $1$, we have the following expansion for the conjugate momentum.
\[ 
\begin{split}
\pi_{rr} & = O(r^{-2}) \\
\pi_{ab} & =O(1) \\
\pi_{ar} & =O(r^{-1}). 
\end{split}
\]
Moreover, we have
\[ 
\begin{split}
g^{rr} & = 1+ O(r^{-1}) \\
g^{ab} & =r^{-2} \tilde \sigma^{ab}+ O(r^{-3}) \\
g^{ar} & = O(r^{-2}).
\end{split}
\]
It is straightforward to verify that 
\[ 
\begin{split}
\nabla_g^i \pi_{ir}  
=&  g^{rr}\partial_r \pi_{rr} + g^{ab} (\partial_a \pi_{br} - \Gamma_{ab}^c \pi_{cr} - \Gamma_{ab}^r \pi_{rr} - \Gamma_{ar}^c \pi_{bc}) + O(r^{-4})\\
=&  \partial_r \pi_{rr} + \nabla^a \pi_{ar} + \frac{2\pi_{rr}}{r} -\frac{g^{ab}\pi_{ab}}{r}+O(r^{-4})
\end{split}
\]
Similarly,
\[ 
\begin{split}
\nabla_g^i \pi_{ia}  
=&  \partial_r \pi_{ar}  - \frac{\pi_{ar}}{r} + \nabla^b_{\sigma} \pi_{ba} + \frac{3\pi _{ar}}{r}+ O(r^{-3})\\
\end{split}
\]
\end{proof}
\begin{lemma}
Let $(M,g,k)$ be an asymptotically flat initial data set of order $1$ satisfying the vacuum constraint equation
and $\pi$ be the conjugate momentum.  Then
\begin{equation}
\begin{split}
\nabla^i_g \pi_{ir} =&  \frac{ \tilde \nabla^a \pi_{ar}^{(-1)}  - \tilde \sigma^{ab}\pi_{ab}^{(0)}  }{r^3}+ O(r^{-4})\\
\nabla^i_g \pi_{ra} =&  \frac{ \pi_{ar}^{(-1)}  - \tilde \nabla^{b}\pi_{ba}^{(0)}  }{r^2}+ O(r^{-3}).
\end{split}
\end{equation}
\end{lemma}
\begin{proof}
For  an asymptotically flat initial data set of order $1$ , we have
\begin{equation}
\begin{split}
\partial_r \pi_{rr} =& \frac{-2 \pi^{(-2)}_{rr}}{r^3}+ O(r^{-4})\\
\partial_r \pi_{ar}  =&  \frac{ - \pi^{(-1)}_{ar} }{r^2}+ O(r^{-3}).
\end{split}
\end{equation}
Combining the above with Lemma A.2 gives the desired result.
\end{proof}


\begin{thebibliography}{99}  
\bibitem{Arnowitt-Deser-Misner} R. Arnowitt, S. Deser\ and\ C. W. Misner, The dynamics of general relativity, in {\it Gravitation: An introduction to current research}, 227--265, Wiley, New York.
\bibitem{Ashtekar-Hansen} A. Ashtekar and R. O. Hansen, \textit{A unified treatment of null and spatial infinity in general relativity. I. Universal structure, asymptotic symmetries, and conserved quantities at spatial infinity,} J. Math. Phys. 19 (1978), no. 7, 1542--1566.

\bibitem{Beig-Omurchadha} R. Beig and N. \'O Murchadha, 
\textit{The Poincar\'e group as the symmetry group of canonical general relativity.}
Ann. Physics 174 (1987), no. 2, 463--498. 


\bibitem{Brown-York1} J. D. Brown and J. W.  York, \textit{Quasilocal energy in
general relativity,} Mathematical aspects of classical field
theory (Seattle, WA, 1991), 129--142, Contemp. Math., 132, Amer.
Math. Soc., Providence, RI, 1992.

\bibitem{Brown-York2} J. D. Brown and J. W. York, \textit{Quasilocal energy and
conserved charges derived from the gravitational action,} Phys.
Rev. D (3) \textbf{47} (1993), no. 4, 1407--1419.


\bibitem{Chen-Huang-Wang-Yau}P.-N. Chen, L.-H. Huang, M.-T. Wang, and S.-T. Yau, \textit{On the validity of the ADM angular momentum}, arXiv:1401.0597.

\bibitem{Chen-Wang-Yau1} P.-N. Chen, M.-T. Wang, and S.-T. Yau, \textit {Evaluating quasilocal energy and solving optimal embedding equation at null infinity,} Comm. Math. Phys. \textbf{308} (2011), no.3, 845--863.

\bibitem{Chen-Wang-Yau2}  P.-N. Chen, M.-T. Wang, and S.-T. Yau, \textit{Minimizing properties of critical points of quasi-local energy,} to appear in Comm. Math. Phys, arXiv:1302.5321
\bibitem{Christodoulou} D. Christodoulou, \textit{Mathematical problems of general relativity. I,} Z\"{u}rich Lectures in Advanced Mathematics. European Mathematical Society (EMS), Z\"{u}rich, 2008. 


\bibitem{Christodoulou-Klainerman} D. Christodoulou and S. Klainerman,  \textit{The global nonlinear stability of the Minkowski space.} Princeton Mathematical Series, 41. Princeton University Press, Princeton, NJ, 1993. 



\bibitem{Chrusciel} P. T. Chru\'{s}ciel, \textit{On the invariant mass conjecture in general relativity,}  Commun. Math. Phys. 120, 233--248 (1988)

\bibitem{Chrusciel2} P. T. Chru\'{s}ciel, \textit{On angular momentum at spatial infinity,} Classical Quantum Gravity 4 (1987), no. 6, L205--L210. 
\bibitem{Corvino-Schoen} J. Corvino and R. Schoen, \textit{On the asymptotics for the vacuum Einstein constraint equations,} J. Diff. Geom. 73(2), 185--217 (2006).

\bibitem{Huang} L.-H.~Huang,
\textit{On the center of mass of isolated systems with general asymptotics,}
Classical Quantum Gravity 26 (2009), no. 1, 015012, 25 pp.

\bibitem{HSW} L.-H. Huang, M.-T. Wang, and R. Schoen, \textit{Specifying angular momentum and center of mass for vacuum initial data sets,} Comm. Math. Phys. 306 (2011), no.3, 785--803, arXiv:1008.4996


\bibitem{Huisken-Yau} G. Huisken and S.-T. Yau, \textit{Definition of center of mass for isolated physical systems and unique foliations by stable spheres with constant mean curvature,}  Invent. Math. 124, 281--311 (1996).

\bibitem{mt} P. Miao and L.-F. Tam, \textit{On second variation of Wang-Yau quasi-local energy,} to appear in Ann. Henri Poincar\'e, arXiv:1301.4656.
\bibitem{Nerz} C. Nerz, \textit{Time evolution of ADM and CMC center of mass in general relativity}, arXiv:1312.6274

\bibitem{n} L. Nirenberg, \textit{The Weyl and Minkowski problems in differential geometry in the large}, Comm. Pure Appl. Math. {\bf 6} (1953), 337--394.
\bibitem{Penrose} R. Penrose, \textit{Some unsolved problems in classical general relativity,} Seminar on Differential Geometry, pp. 631--668,
Ann. of Math. Stud., 102, Princeton Univ. Press, Princeton, N.J., 1982
\bibitem{Penrose2} R. Penrose, \textit{Quasi-local mass and angular momentum in general relativity,} Proc. Roy. Soc. London Ser. A 381 (1982), no. 1780, 53--63.


\bibitem{Regge-Teitelboim} T. Regge and C. Teitelboim, \textit{Role of Surface Integrals in the Hamiltonian Formulation of General Relativity,} Ann. Phys. 88, 286--318 (1974).
\bibitem{Schoen-Yau1} R. ~Schoen and S.-T. ~Yau,  \textit{Positivity of the total mass
of a general space-time,} Phys. Rev. Lett. \textbf{43} (1979), no.
20, 1457--1459

\bibitem{Schoen-Yau2} R. Schoen and S.-T. Yau, \textit{On the proof of the positive
mass conjecture in general relativity,} Comm. Math. Phys.
\textbf{65} (1979), no. 1, 45--76.

\bibitem{Szabados} L. B. Szabados, \textit{Quasi-local energy-momentum and angular momentum in general relativity,}
\bibitem{Wald} R. M. Wald, \textit{General relativity,} University of Chicago Press, Chicago, IL, 1984. xiii+491 pp.
\bibitem{Wang-Yau1} M.-T. Wang and S.-T. Yau, \textit{Quasilocal mass in general relativity,} Phys. Rev. Lett. \textbf{102} (2009), no. 2, no. 021101.
\bibitem{Wang-Yau2} M.-T. Wang and S.-T. Yau, \textit{Isometric embeddings into the Minkowski space and new quasi-local mass,} Comm. Math. Phys. \textbf{288} (2009), no. 3, 919--942.
\bibitem{Wang-Yau3} M.-T. Wang and S.-T. Yau, \textit{Limit of quasilocal mass at spatial infinity,} Comm. Math. Phys. \textbf{296} (2010), no.1, 271--283. arXiv:0906.0200v2.
\bibitem{Witten} E. Witten, \textit{A new proof of the positive energy theorem,}
Comm. Math. Phys. \textbf{80} (1981), no. 3, 381--402.


\end{thebibliography}
\end{document}